\documentclass[a4paper,reqno,11pt,draft]{amsart}
\usepackage{verbatim}
\usepackage{amssymb,amsmath,amsthm}
\usepackage{amsfonts}
\usepackage{array}

\def\HH{\overline{H}}

\newtheorem{theorem}{Theorem}[section]
\newtheorem{lemma}{Lemma}[section]
\newtheorem{corollary}{Corollary}[section]

\newtheorem{remark}{{\rm\bf Remark}}[theorem]

\begin{document}

\title[New properties of multiple harmonic sums]{New properties of multiple harmonic sums
modulo $p$ and $p$-analogues of Leshchiner's series}

\author{Kh. Hessami Pilehrood}
\address{Department of Mathematics and Statistics \\ Dalhousie University \\ Halifax, Nova Scotia, B3H 3J5, Canada}
\email{hessamik@gmail.com}
\thanks{}

\author{T. Hessami Pilehrood}
\address{Department of Mathematics and Statistics \\ Dalhousie University \\ Halifax, Nova Scotia, B3H 3J5, Canada}
\email{hessamit@gmail.com}
\thanks{}

\author{R. Tauraso}

\address{Dipartimento di Matematica \\
Universit\`a di Roma ``Tor Vergata'' \\ 00133 Roma, Italy}
\email{tauraso@mat.uniroma2.it}
\thanks{}


\date{}

\begin{abstract}
In this paper we present some new
binomial identities for multiple harmonic sums
whose indices are the sequences $(\{1\}^a,c,\{1\}^b),$ $(\{2\}^a,c,\{2\}^b)$
and prove a number of congruences for these sums modulo a prime $p.$
The congruences obtained allow us to find nice $p$-analogues of Leshchiner's
series for zeta values and to refine a result due to M.~Hoffman and J.~Zhao
about the set of generators of the multiple harmonic sums of weight $7$ and $9$ modulo~$p$.
As a further application we provide a new proof of
Zagier's formula for $\zeta^{*}(\{2\}^a,3,\{2\}^b)$ based on a finite identity
for partial sums of the zeta-star series.
\end{abstract}

\maketitle

\section{Introduction}
\label{intro}

In the last few years there has been a growing attention in the study of $p$-adic analogues of
various binomial series related to multiple zeta values, which are nested generalizations of the classical
Riemann zeta function $\zeta(s)=\sum_{n=1}^s1/n^s.$ The main reason of interest of such $p$-analogues is that they are related
to divisibility properties of multiple harmonic sums
which can be considered as elementary ``bricks'' for expressing complicated congruences.
Before discussing this further we recall the precise definitions of such objects.

For $r\in {\mathbb N},$  ${\bf s}=(s_1, s_2, \ldots, s_r)\in ({\mathbb Z}^{*})^r,$ and
a non-negative integer $n,$  the {\it alternating multiple harmonic sum} is defined by
\begin{equation*}
H_n(s_1, s_2, \ldots, s_r)=\sum_{1\le k_1<k_2<\ldots<k_r\le n}\prod_{i=1}^r
\frac{\textup{sgn}\,(s_i)^{k_i}}{k_i^{|s_i|}}
\end{equation*}
and the {\it ``odd'' alternating multiple harmonic sum} is given by
\begin{equation*}
\HH_n(s_1, s_2, \ldots, s_r)=\sum_{0\le k_1<k_2<\ldots<k_r< n}\prod_{i=1}^r
\frac{\textup{sgn}\,(s_i)^{k_i}}{(2k_i+1)^{|s_i|}}.
\end{equation*}
If all $s_1, \ldots, s_r$ are positive, then $H_n(s_1,\ldots, s_r)$ and $\HH_n(s_1, s_2, \ldots, s_r)$
are called {\it multiple harmonic sum} (MHS) and {\it ``odd'' multiple harmonic sum,} respectively.
For ${\bf s}=(s_1, \ldots, s_r)\in {\mathbb N}^r,$ it is also convenient to consider the {\it  multiple star harmonic  sum} (or non-strict MHS)
\begin{equation*}
S_n(s_1,\ldots,s_r)=\sum_{1\le k_1\le\cdots\le k_r\le n}\frac{1}{k_1^{s_1}\cdots k_r^{s_r}}.
\end{equation*}
The integers  $l({\bf s}):=r$ and $w:=|{\bf s}|:=\sum_{i=1}^r|s_i|$ are called the length (or depth) and the weight
of a multiple harmonic sum.
By convention, we put  $H_n({\bf s})=\HH_n({\bf s})=0$ if $n<r,$
and $H_n(\emptyset)=\HH_n(\emptyset)=S_n(\emptyset)=1.$ By $\{s_1, s_2, \ldots, s_j\}^m$ we denote the
sequence of length $mj$ with $m$ repetitions of $(s_1, s_2, \ldots, s_j)$.

MHSs occur naturally in different areas of mathematics such as combinatorics, number theory, algebraic geometry, quantum groups and knot theory.
For a long time, MHSs have been of interest to physicists \cite{Bl:99, V:99}, who realized these sums
as Mellin transforms of some functions occurring in quantum field theories.
The limit cases of MHSs give rise to multiple zeta values (MZVs):
\begin{equation*}
\zeta(s_1, s_2, \ldots, s_r)=\lim_{n\to\infty}H_n(s_1, s_2, \ldots, s_r),
\end{equation*}
\begin{equation*}
\zeta^{*}(s_1, s_2, \ldots, s_r)=\lim_{n\to\infty}S_n(s_1, s_2, \ldots, s_r)
\end{equation*}
defined for $s_1, \ldots s_{r-1}\ge 1$ and $s_r\ge 2$ to ensure convergence of the series.

The earliest results on MZVs are due to Euler who
elaborated a method to reduce double sums of small weight to certain rational linear combinations of products of single sums.
In particular, he proved the simple but non-trivial relation $\zeta(1,2)=\zeta(3)$ and determined
the explicit values of the zeta function at even integers:
\begin{equation*}
\zeta(2m)=\frac{(-1)^mB_{2m}}{2(2m)!}\,(2\pi)^{2m},
\end{equation*}
where $B_k\in {\mathbb Q}$ are the Bernoulli numbers defined by the generating function
\begin{equation*}
\frac{x}{e^x-1}=\sum_{k=0}^{\infty}B_k \frac{x^k}{k!}.
\end{equation*}
It is easy
to verify that $B_0=1,$ $B_1=-1/2,$ $B_2=1/6,$ $B_4=-1/30,$ and $B_{2m+1}=0$ for $m\ge 1.$
The denominators of the Bernoulli numbers $B_{2m}$ are completely characterized
by the Clausen--von Staudt Theorem \cite[p.~233]{Ir}, which states that
\begin{equation*}
\text{for} \quad m\in{\mathbb N}, \quad B_{2m}+\sum_{p-1|2m}\frac{1}{p} \quad\text{is an integer}.
\end{equation*}
Early results for the values modulo $p$ of the multiple harmonic sums, when $p$ is a prime greater than $3,$
go back to the classical Wolstenholme's Theorem \cite{W}:
\begin{equation*}
H_{p-1}(1)\equiv 0 \pmod{p^2}, \qquad H_{p-1}(2)\equiv 0 \pmod{p}.
\end{equation*}
Glaisher \cite{G} in 1900, and Lehmer \cite{L:38} in 1938, proved that
even the multiple harmonic sums $H_{p-1}(m)$ modulo a higher power of a prime $p\ge m+3,$
are related to the Bernoulli numbers:
\begin{equation*}
H_{p-1}(m)\equiv \begin{cases}
\frac{m(m+1)}{2(m+2)}\,p^2\, B_{p-m-2}\pmod{p^3}     & \quad \text{if} \quad m \quad \text{is odd}, \\[3pt]
\frac{m}{m+1}\,p\,B_{p-m-1} \quad\,\,\,\, \pmod{p^2} & \quad\text{if} \quad m \quad \text{is even}.
\end{cases}
\end{equation*}

The systematic study of MZVs began in the early 1990s with the works of Hoffman \cite{Ho:92}
and Zagier \cite{Za:92}.
The set of the MZVs has a rich algebraic structure given by the shuffle and the stuffle (harmonic shuffle or quasi-shuffle) relations.
These follow from the representation of multiple zeta values in terms of iterated integrals and harmonic sums,
respectively.
There are many conjectures concerning multiple zeta values and
despite some recent progress, lots of open questions still remain to be answered.

Let ${\mathfrak{Z}}_w$ denote the ${\mathbb Q}$-vector space spanned by the set of multiple zeta values $\zeta(s_1,\ldots,s_r)$
with $s_r\ge 2$ and the total weight $w=s_1+\cdots+s_r,$ and let ${\mathfrak Z}$ denote the ${\mathbb Q}$-vector space spanned by all
multiple zeta values over ${\mathbb Q}.$ 
A conjecture of Zagier \cite{Za:92} states that the dimension of the ${\mathbb Q}$-vector space
${\mathfrak{Z}}_w$ is given by the Perrin numbers $d_w$ defined for $w\ge 3$ by the recurrence
\begin{equation*}
d_w=d_{w-2}+d_{w-3}
\end{equation*}
with the initial conditions $d_0=1,$ $d_1=0,$ $d_2=1.$ The upper bound $\dim{\mathfrak Z}_w\le d_w$ was proved independently by Goncharov \cite{Go} and Terasoma \cite{Ter}.

It is easy to see that the Perrin number $d_w$ is equal to the number of multiple zeta values $\zeta(s_1,\ldots,s_r)$ with
$s_1+\cdots+s_r=w$ and each $s_j\in\{2,3\}.$
While investigating the deep algebraic structure of ${\mathfrak Z}$, Hoffman \cite{Ho:97}
conjectured that the MZVs $\zeta(s_1, \ldots, s_r)$ of weight $w$ with $s_j\in\{2,3\}$ span the ${\mathbb Q}$-space ${\mathfrak Z}_w.$
Very recently, this conjecture was proved using motivic ideas by Brown~\cite{Brown1}.
So the main problem which remains open is proving that the numbers $\zeta(s_1,\ldots,s_r)$ with $s_j\in\{2,3\}$
are linearly independent over ${\mathbb Q}.$

According to Zagier's  conjecture, a basis for ${\mathfrak Z}_w$ for $2\le w\le 9$ should be given as follows (see \cite{Wa}):
\begin{equation*}
\begin{array}{lll}
w=2, &  d_2=1, & \zeta(2); \\
w=3, &  d_3=1, & \zeta(3); \\
w=4, &  d_4=1, & \zeta(2,2); \\
w=5, &  d_5=2, & \zeta(2,3), \zeta(3,2); \\
w=6, &  d_6=2, & \zeta(2,2,2), \zeta(3,3);\\
w=7, &  d_7=3, & \zeta(2,2,3), \zeta(2,3,2), \zeta(3,2,2); \\
w=8, &  d_8=4, & \zeta(2,2,2,2), \zeta(2,3,3), \zeta(3,2,3), \zeta(3,3,2); \\
w=9, &  d_9=5, & \zeta(2,2,2,3), \zeta(2,2,3,2), \zeta(2,3,2,2), \zeta(3,2,2,2), \zeta(3,3,3).
\end{array}
\end{equation*}
Brown \cite{Brown1} proved that
part of this conjectural basis, namely the multiple zeta values of the form
$\zeta(\{2\}^a,3,\{2\}^b)$ can be expressed in terms of ordinary zeta values.
The exact formulae were found and proved by Zagier \cite{Za:12}:
\begin{equation*}
\zeta(\{2\}^a,3,\{2\}^b)=2\sum_{r=1}^K(-1)^r\left(\binom{2r}{2a+2}-(1-2^{-2r})\binom{2r}{2b+1}\right)\zeta(2r+1)H(k-r)
\end{equation*}
and
\begin{equation} \label{zzzz}
\zeta^*(\{2\}^a,3,\{2\}^b)=-2\sum_{r=1}^K\!\left(\binom{2r}{2a}-\delta_{r,a}-(1-2^{-2r})\binom{2r}{2b+1}\right)\zeta(2r+1)H^*(K-r),
\end{equation}
where $K=a+b+1$ and
\begin{equation*}
H(n)=\zeta(\{2\}^n)=\frac{\pi^{2n}}{(2n+1)!}, \quad H^*(n)=\zeta^*(\{2\}^n)=2(1-2^{1-2n})\zeta(2n), \,\, n\ge 0.
\end{equation*}
The theory of multiple harmonic sums modulo $p$ bears many similarities with the theory of multiple zeta values.
In the early 2000s Zhao \cite{Zhao:06, Zhao:12} began to generalize the Wolstenholme theorem and Glaisher--Lehmer congruences to other multiple harmonic sums with special emphasis to the cases where the sums are divisible by a high power of a prime.
Later, Hoffman \cite{Ho:03} showed that the algebraic setup developed in order to study MZVs can be extended
to deal with MHSs as well.
In \cite{Hof:07}, Hoffman described the possible relations modulo $p$ between MHSs of given weight
not exceeding~$9.$

For a positive integer $w,$ let $c_w$ denote the minimal number of  harmonic sums of weight $w$
which are needed to generate all MHSs of weight $w$ modulo $p$ for $p>w+1.$
For the first few values of $c_w,$ Hoffman \cite{Hof:07} obtained the following table:
\begin{equation*}
\begin{tabular}{|c|c|c|c|c|c|c|c|c|c|}
  \hline
    $w$ & 1 & 2 & 3 & 4 & 5 & 6 & 7 & 8 & 9 \\
  \hline
  $c_w$ & 0 & 0 & 1 & 0 & 1 & 1 & 2 & 2 & 2 \\
  \hline
\end{tabular}
\end{equation*}
The value $c_9=2$ was calculated conditionally under the assumption of the congruence
\begin{equation} \label{ccc}
S_{p-1}(1,1,1,6)\equiv \frac{1}{54}\,B_{p-3}^3+\frac{1889}{648}\,B_{p-9} \pmod{p},
\end{equation}
which was first conjectured and verified for all primes $10<p<2000$ by Zhao \cite[Prop.~3.2]{Zhao:06}.

The purpose of the present paper is to study multiple harmonic sums of the form
\begin{equation*}
S_n(\{2\}^a,c,\{2\}^b), \quad S_n(\{1\}^a,c,\{1\}^b), \qquad c=1,2,3,\ldots.
\end{equation*}
We  prove  some new
binomial identities for these sums and apply them to obtain congruences modulo a prime $p$ for
\begin{equation} \label{abc}
S_{p-1}(\{2\}^a,1,\{2\}^b), \quad S_{p-1}(\{2\}^a,3,\{2\}^b), \quad S_{p-1}(\{1\}^a,2,\{1\}^b).
\end{equation}
This allows us to give a new proof of Zagier's formula (\ref{zzzz}) for
$\zeta^{*}(\{2\}^a,3,\{2\}^b)$
based on a finite identity for $S_n(\{2\}^a,3,\{2\}^b)$
and to formulate its finite $p$-analogue.
In addition, those congruences of the finite sums (\ref{abc})
sharpen Hoffman's and Zhao's results on MHSs of small weights. Indeed,
we show that in weight $7$ the number of generators $c_7$ equals $1$, and in weight $9$ we prove congruence (\ref{ccc}), which implies the equality $c_9=2$ unconditionally.
From \cite[Section 10]{Hof:07} and Corollaries \ref{c3} and \ref{c4.2} below, the set of generators modulo $p$ for multiple harmonic sums
$S_{p-1}({\bf s})$ of weight $w=|{\bf s}|\le 9$  when $p>w+1$ in terms of Bernoulli numbers is as follows:
\begin{equation*}
\begin{array}{lll}
w=1, &  c_1=0, & 0; \\
w=2, &  c_2=0, & 0; \\
w=3, &  c_3=1, & S_{p-1}(1,2)\equiv B_{p-3}; \\
w=4, &  c_4=0, & 0; \\
w=5, &  c_5=1, & S_{p-1}(1,4)\equiv B_{p-5};\\
w=6, &  c_6=1, & S_{p-1}(1,1,4)\equiv -\frac{1}{6}\,B_{p-3}^2; \\
w=7, &  c_7=1, & S_{p-1}(1,6)\equiv B_{p-7}; \\
w=8, &  c_8=2, & S_{p-1}(1,1,6), \,\, S_{p-1}(1,4)S_{p-1}(1,2)\equiv B_{p-5}B_{p-3}; \\
w=9, &  c_9=2, & S_{p-1}(1,8)\equiv B_{p-9}, \,\, S_{p-1}(1,2)S_{p-1}(1,1,4)\equiv -\frac{1}{6}\,B_{p-3}^3.
\end{array}
\end{equation*}
Hoffman \cite{Hof:07} conjectured that all multiple harmonic sums $S_{p-1}({\bf s})$ (or $H_{p-1}({\bf s})$) can be written
modulo $p$ as sums of products of ``height one sums'' $S_{p-1}(1,\ldots,1,2h)$ (or $H_{p-1}(1,\ldots,1,2h)$) with rational coefficients.
Note that this conjecture has been confirmed only for sums of weight~$w$ with $w\le 9.$
Zhao \cite{Zhao:11} continued
computations of sets of generators for $w=10, 11, 12.$ Using reduction, stuffle and duality relations he found
that the MHSs of weight $10$ can be generated by the set
\begin{equation*}
B_{p-3}B_{p-7}\equiv S_{p-1}(1,2)S_{p-1}(1,6), \,\,\,\, B_{p-5}^2\equiv S_{p-1}^2(1,4),
\end{equation*}
\begin{equation*}
S_{p-1}(1,1,8), \,\,\,
S_{p-1}(1,1,1,1,6), \,\,\, S_{p-1}(2,2,1,4,1),
\end{equation*}
which according to  Hoffman's conjecture still contains the extra term $S_{\!p-1}\!(2,\!2,\!1,\!4,\!1).$
Zhao \cite{Zhao:12a} kindly communicated us that using our Theorems \ref{T3.3.1} and \ref{T3.3.2} (see below) it is possible to reduce
the set of $8$ generators for  MHSs of weight $11$ found in \cite{Zhao:11} to the set of $5$ elements:
\begin{equation*}
B_{p-11}\equiv S_{p-1}(1,10), \, B_{p-5}B_{p-3}^2, \, B_{p-3}S_{p-1}(1,1,6), \, S_{p-1}(1,1,1,8), \, S_{p-1}(5,3,2,1).
\end{equation*}
In contrast to the corresponding conjecture for MZVs, not all ``height one sums'' are
claimed to be linearly independent modulo $p$, because
they satisfy the duality relation \cite[Theorem~5.2]{Hof:07}:
\begin{equation*}
S_{p-1}(\{1\}^{k-1},h)\equiv (-1)^{k+h}S_{p-1}(\{1\}^{h-1},k) \pmod{p}, \quad p>\max(h,k).
\end{equation*}
In \cite[Remark 2.3]{Zhao:11}, Zhao conjectured that weight-$8$ generators $S_{p-1}(1,1,6)$ and $B_{p-3}B_{p-5}$ should
be linearly independent modulo $p$ over ${\mathbb Q}(10),$ where
\begin{equation*}
{\mathbb Q}(w):=\{a/b\in {\mathbb Q}: a/b \,\,\, \text{is reduced and if a prime}\,\,\, q|b \,\,\, \text{then} \,\,\, q\le w\}.
\end{equation*}
Zhao \cite{Zhao:11} also formulated a general conjecture on linear independence modulo $p$ of products of Bernoulli numbers
\begin{equation*}
B_{p-i_1}B_{p-i_2}\cdots B_{p-i_r}
\end{equation*}
over ${\mathbb Q}(w+2),$
where $i_1, \ldots, i_r$ are odd indices  greater than $1$ with $i_1+\ldots+i_r=w,$ which in particular,
implies the linear independence of weight-$9$ generators $B_{p-3}^3$ and $B_{p-9}$ modulo $p$ over ${\mathbb Q}(11).$

As an application
of our results to ordinary zeta values, we consider generalizations of
Ap\'ery's famous  series
\begin{equation} \label{eq01}
\zeta(2)=3\sum_{k=1}^{\infty}\frac{1}{k^2\binom{2k}{k}}, \qquad
\zeta(3)=\frac{5}{2}\sum_{k=1}^{\infty}\frac{(-1)^{k-1}}{k^3\binom{2k}{k}},
\end{equation}
used in his irrationality proof of $\zeta(2)$ and $\zeta(3).$
Among different extensions \cite{Le:81, BB:97, BBB:2006, HH:08, HH:2011} of series (\ref{eq01}),
we focus on those obtained by Leshchiner \cite{Le:81} for values of the Riemann zeta function $\zeta(s)$  and Dirichlet beta function $\beta(s)=\sum_{n=0}^{\infty}(-1)^n/(2n+1)^s.$
Note that the value  $\beta(2m+1)$ as well as $\zeta(2m+2)$ can be expressed in terms of $\pi,$
\begin{equation*}
\beta(2m+1)=\frac{(-1)^{m}E_{2m}}{2^{2m+2}(2m)!}\,\pi^{2m+1},
\end{equation*}
where $m$ is a non-negative integer and $E_{2m}\in {\mathbb Q}$ are  Euler numbers.

Leshchiner established the following four expansions in terms of
multiple harmonic sums that generalize series (\ref{eq01}):
\begin{equation}\label{eq02}
\begin{split}
\qquad\qquad \Bigl(1-\frac{1}{2^{2m+1}}\Bigr)\zeta(2m+2)&=\frac{3}{2}\sum_{k=1}^{\infty}\frac{(-1)^m H_{k-1}(\{2\}^m)}{k^2\binom{2k}{k}} \\[3pt]
&\quad+2\sum_{j=1}^m\sum_{k=1}^{\infty}\frac{(-1)^{m-j}H_{k-1}(\{2\}^{m-j})}{k^{2j+2}\binom{2k}{k}},
\end{split}
\end{equation}
\begin{equation}\label{eq03}
\begin{split}
\qquad\qquad (-1)^{m-1}\cdot\zeta(2m+3)&=\frac{5}{2}\sum_{k=1}^{\infty}\frac{(-1)^{k}H_{k-1}(\{2\}^m)}{k^3\binom{2k}{k}} \\[3pt]
&\quad+2\sum_{j=1}^m\sum_{k=1}^{\infty}\frac{(-1)^{k-j}H_{k-1}(\{2\}^{m-j})}{k^{2j+3}\binom{2k}{k}},
\end{split}
\end{equation}
\begin{equation}
\begin{split}
\qquad\qquad\Bigl(1-\frac{1}{2^{2m+2}}\Bigr)\zeta(2m+2)&=\frac{5}{4}\sum_{k=0}^{\infty}\frac{(-1)^{k+m}\binom{2k}{k}\HH_k(\{2\}^m)}{16^k(2k+1)^2}
\\[3pt]
&+\sum_{j=1}^m\sum_{k=0}^{\infty}\frac{(-1)^{k+m-j}\binom{2k}{k}\HH_k(\{2\}^{m-j})}{16^k(2k+1)^{2j+2}}, \label{eq05}
\end{split}
\end{equation}
\begin{equation}\label{eq04}
\beta(2m+1)=\frac{3}{4}\sum_{k=0}^{\infty}\frac{(-1)^m\binom{2k}{k}\HH_k(\{2\}^m)}{16^k(2k+1)}
+\sum_{j=1}^m\sum_{k=0}^{\infty}\frac{(-1)^{m-j}\binom{2k}{k}\HH_k(\{2\}^{m-j})}{16^k(2k+1)^{2j+1}}.
\end{equation}
Indeed,  if we put $m=0$ in (\ref{eq02}), (\ref{eq03}), we get Ap\'ery's series (\ref{eq01}).

Recently, the authors \cite{Ta:10, HH:11, HHT:11},  showed that the series (\ref{eq02}), (\ref{eq03}) for $m=0, 1$
and the series (\ref{eq05}), (\ref{eq04})  for $m=0$ admit very nice $p$-analogues. Indeed, if we truncate
the series (\ref{eq02}), (\ref{eq03}) with $m=0$ or $1$ up to $p-1$ and the series (\ref{eq05}), (\ref{eq04})  with $m=0$ up to $(p-3)/2,$
where $p$ is an odd prime, we get congruences for the finite sums modulo  powers of $p$ expressible in terms of Bernoulli numbers.
In this paper, we extend these results to any non-negative integer $m.$
This is done by showing that the finite $p$-analogues of the series (\ref{eq02})--(\ref{eq04}) are related with
MHSs of the form
\begin{equation*}
H_{p-1}(\{2\}^m,3), \quad H_{p-1}(\{2\}^m,1), \quad \HH_{\frac{p-1}{2}}(\{2\}^m,3), \quad \HH_{\frac{p-1}{2}}(\{2\}^m,1).
\end{equation*}

\section{Identities for multiple harmonic sums}

We start with a list of binomial identities that we will need later. Note that a generalization of \eqref{eq3.2.2} appears in
\cite[(4.20)]{Go:72}

\begin{lemma} \label{l3.2.1}
For any positive integers $m, n$ and a non-negative integer $l$ we have
\begin{equation} \label{eq3.2.1}
\sum_{k=l+1}^n(-1)^{k-1}\binom{mn}{n-k}=
(-1)^l\binom{mn-1}{n-l-1},
\end{equation}
\begin{equation}
2\sum_{k=l+1}^n\frac{k\binom{n}{k}}{\binom{n+k}{k}}=
\frac{n\binom{n-1}{l}}{\binom{n+l}{l}}, \label{eq3.2.1b}
\end{equation}
\begin{equation}
2\sum_{k=1}^n\frac{\binom{n}{k}}{k\binom{n+k}{k}}=\sum_{k=1}^n\frac{1}{k}.
\label{eq3.2.1d}
\end{equation}
If $l\ge 1,$ then
\begin{equation}
\sum_{k=l}^n\frac{\binom{k}{l}}{k^2\binom{k+l}{l}}=\frac{\binom{n}{l}}{l^2\binom{n+l}{l}},
\label{eq3.2.1c}
\end{equation}
If $n\ge 2,$ then
\begin{equation} \label{eq3.2.2}
\sum_{k=1}^n\frac{(-1)^kk^2\binom{n}{k}}{\binom{n+k}{k}}=0.
\end{equation}
\end{lemma}
\begin{proof}
For the proof of our identities it is useful to consider the usual binomial coefficient $\binom{r}{k}$
in a more general setting, that is, to allow an arbitrary integer to appear in the lower index of $\binom{r}{k}.$
For this purpose we set
\begin{equation*}
\binom{r}{k}=\begin{cases}
\frac{r(r-1)\cdots (r-k+1)}{k(k-1)\cdots 1}  &\qquad\text{if integer}\quad k\ge 0,\\
0  &\qquad\text{if integer}\quad k< 0.
\end{cases}
\end{equation*}
To prove the first identity we first observe that
\begin{equation}
(-1)^{k-1}\binom{mn}{n-k}=G(n,k+1)-G(n,k)
\label{eq22}
\end{equation}
for positive integers $n,k$ with
\begin{equation*}
G(n,k)=(-1)^k\binom{mn-1}{n-k}.
\end{equation*}
Then summing both sides of equation (\ref{eq22}) over $k$ from $l+1$ to $n$ we obtain
\begin{equation*}
\sum_{k=l+1}^n\!(-1)^{k-1}\binom{mn}{n-k}=G(n,n+1)-G(n,l+1)=-G(n,l+1)=
(-1)^l\binom{mn-1}{n-l-1}.
\end{equation*}
Similarly, for the proof of the second identity we have
\begin{equation} \label{eq23}
\frac{2k\binom{n}{k}}{\binom{n+k}{k}}=G(n,k+1)-G(n,k)
\end{equation}
for positive integers $n,k$ with
\begin{equation*}
G(n,k)=-\frac{(n+k)\binom{n}{k}}{\binom{n+k}{k}}.
\end{equation*}
Summing both sides of equation (\ref{eq23}) over $k$ from $l+1$ to $n$ we easily obtain the result.

To prove  identity (\ref{eq3.2.1d}) we set
\begin{equation*}
f(m,k):=\frac{\binom{m}{k}}{k\binom{m+k}{k}} \qquad\text{and} \qquad
G(m,k):=-\frac{\binom{m+1}{k}}{2(m+1)\binom{m+k}{k}}, \quad m\ge 0, \,k\ge 1.
\end{equation*}
Then it is easy to see that
\begin{equation} \label{eq25}
f(m+1,k)-f(m,k)=G(m,k+1)-G(m,k).
\end{equation}
Summing both sides of equation (\ref{eq25}) over $m$ from $0$ to $n-1$ we have
\begin{equation*}
f(n,k)=\sum_{m=0}^{n-1}(G(m,k+1)-G(m,k)).
\end{equation*}
Now summing the above equation once again over $k$ from $1$ to $n$ we obtain
\begin{equation*}
\begin{split}
\sum_{k=1}^nf(n,k)=\sum_{m=0}^{n-1}\sum_{k=1}^n&(G(m,k+1)-G(m,k))=
\sum_{m=0}^{n-1}(G(m,n+1)-G(m,1)) \\
&=-\sum_{m=0}^{n-1}G(m,1)=\frac{1}{2}\sum_{m=0}^{n-1}\frac{1}{m+1},
\end{split}
\end{equation*}
as required.

For proving identity (\ref{eq3.2.1c}), it is easy to see  that
\begin{equation} \label{eq24}
\frac{l^2\binom{k}{l}}{k^2\binom{k+l}{l}}=G(k+1,l)-G(k,l)
\end{equation}
for positive integers $k,l$ with
\begin{equation*}
G(k,l)=\frac{(k^2-l^2)\binom{k}{l}}{k^2\binom{k+l}{l}}.
\end{equation*}
Then summing both sides of equation (\ref{eq24}) over $k$ from $l$ to $n$ we get
\begin{equation*}
l^2\sum_{k=l}^n\frac{\binom{k}{l}}{k^2\binom{k+l}{l}}=G(n+1,l)-G(l,l)=G(n+1,l)
=\frac{\binom{n}{l}}{\binom{n+l}{l}}.
\end{equation*}

To prove (\ref{eq3.2.2}) we note that for integers $n\ge 2,$ $k\ge 0,$
\begin{equation} \label{eq3.2.3}
\frac{(-1)^{k}k^2\binom{n}{k}}{\binom{n+k}{k}}=G(n,k+1)-G(n,k),
\end{equation}
with
\begin{equation*}
G(n,k)=\frac{(-1)^{k-1}k(k-1)(n+k)\binom{n}{k}}{2(n-1)\binom{n+k}{k}}.
\end{equation*}
Then summing both sides of equation (\ref{eq3.2.3}) over $k$ from $1$ to $n$ we obtain
\begin{equation*}
\sum_{k=1}^n\frac{(-1)^{k}k^2\binom{n}{k}}{\binom{n+k}{k}}=G(n,n+1)-G(n,1)=0,
\end{equation*}
and the lemma is proved.
\end{proof}

\begin{lemma} \label{l3.2.2}
Let $k, m, n$ be positive integers,
$
A_{n,k}^{(m)}=(-1)^{k}\binom{mn}{n-k}c_n^{(m)}
$
where $c_n^{(m)}$ is an arbitrary sequence  independent of $k,$ and $a$ be a non-negative integer.
Then for each   $c\in {\mathbb N},$
\begin{align}\label{id1}
\frac{1}{n^c}\sum_{k=1}^n \frac{H_{k-1}({\bf b})A_{n,k}^{(m)}}{k^a}
=\sum_{k=1}^n \frac{H_{k-1}({\bf b})A_{n,k}^{(m)}}{k^{a+c}}
+\underset{j\ge 0, s_1>a}{\sum_{j+|{\bf s}|=a+c}}
 m^{l({\bf s})}
\sum_{k=1}^n\frac{H_{k-1}({\bf b},{\bf s})A_{n,k}^{(m)}}{k^{j}}
\end{align}
where ${\bf s}=(s_1,s_2,\dots,s_r)\in \mathbb{N}^r$ for $r\geq 0$, and $|{\bf s}|=\sum_{i=1}^r s_i$, $l({\bf s})=r$.
\end{lemma}
\begin{proof} We first note that by (\ref{eq3.2.1}),
\begin{align}\label{c1}
\frac{m}{l}\sum_{k=l+1}^n A_{n,k}^{(m)}=\left(\frac{1}{n}-\frac{1}{l}\right)A_{n,l}^{(m)}.
\end{align}
We prove \eqref{id1} by induction on $c$. By \eqref{c1}, we have that
\begin{align*}
m\sum_{k=1}^n{H_{k-1}({\bf b},a+1)A_{n,k}^{(m)}}
&=m\sum_{l=1}^n \frac{H_{l-1}({\bf b})}{l^{a+1}}\sum_{k=l+1}^n A_{n,k}^{(m)}\\
&=\sum_{l=1}^n\frac{H_{l-1}({\bf b}) A_{n,l}^{(m)}}{l^{a}}\,\left(\frac{1}{n}-\frac{1}{l}\right)\\
&=\frac{1}{n}\sum_{l=1}^n\frac{H_{l-1}({\bf b}) A_{n,l}^{(m)}}{l^{a}}-\sum_{l=1}^n\frac{H_{l-1}({\bf b}) A_{n,l}^{(m)}}{l^{a+1}}
\end{align*}
which is \eqref{id1} for $c=1$.
Now let $c>1$. Then by  induction, we have
\begin{align*}
\frac{1}{n^c}\sum_{k=1}^n &\frac{H_{k-1}({\bf b})A_{n,k}^{(m)}}{k^a}
=\frac{1}{n^{c-1}}\left(\frac{1}{n}\sum_{k=1}^n \frac{H_{k-1}({\bf b})A_{n,k}^{(m)}}{k^a}\right)\\
&\quad=\frac{1}{n^{c-1}}\sum_{k=1}^n \frac{H_{k-1}({\bf b})A_{n,k}^{(m)}}{k^{a+1}}+\frac{m}{n^{c-1}}\sum_{k=1}^n{H_{k-1}({\bf b},a+1)A_{n,k}^{(m)}}\\[3pt]
&\quad=\sum_{k=1}^n \frac{H_{k-1}({\bf b})A_{n,k}^{(m)}}{k^{a+c}}+
\underset{j\ge 0, s_1>a+1}{\sum_{j+|{\bf s}|=a+c}}
 m^{l({\bf s})}
\sum_{k=1}^n\frac{H_{k-1}({\bf b},{\bf s})A_{n,k}^{(m)}}{k^{j}}\\[3pt]
&\quad\,+m\sum_{k=1}^n \frac{H_{k-1}({\bf b},a+1)A_{n,k}^{(m)}}{k^{c-1}}
+m\!\!\underset{j\ge 0, s_1>0}{\sum_{j+|{\bf s}|=c-1}}
 m^{l({\bf s})}
\sum_{k=1}^n\frac{H_{k-1}({\bf b},a+1,{\bf s})A_{n,k}^{(m)}}{k^{j}}\\
&\quad=\sum_{k=1}^n \frac{H_{k-1}({\bf b})A_{n,k}^{(m)}}{k^{a+c}}+
\underset{j\ge 0, s_1>a+1}{\sum_{j+|{\bf s}|=a+c}}
 m^{l({\bf s})}
\sum_{k=1}^n\frac{H_{k-1}({\bf b},{\bf s})A_{n,k}^{(m)}}{k^{j}}\\
&\quad\,+\underset{j\ge 0, s_1=a+1}{\sum_{j+|{\bf s}|=a+c}}
 m^{l({\bf s})}
\sum_{k=1}^n\frac{H_{k-1}({\bf b},{\bf s})A_{n,k}^{(m)}}{k^{j}}.
\end{align*}
\end{proof}

\begin{theorem} \label{T3.2.1}
For a positive integer $n$ and non-negative integers $a, b,$ and $c\geq 2$,
\begin{equation}\label{S2ac2b}
\begin{split}
S_n(\{2\}^a,c,\{2\}^b)&=
2\sum_{k=1}^n \frac{(-1)^{k-1}\binom{n}{k}}{k^{2a+2b+c}\binom{n+k}{k}} \\[3pt]
&+4\!\!\!\!\underset{i\geq 1,j\geq 2, |{\bf s}|\ge 0}{\sum_{i+j+|{\bf s}|=c}}
2^{l({\bf s})}
\sum_{k=1}^n\frac{H_{k-1}(2a+i,{\bf s})(-1)^{k-1}\binom{n}{k}}{k^{2b+j}\binom{n+k}{k}},
\end{split}
\end{equation}
where ${\bf s}=(s_1,s_2,\dots,s_r)\in \mathbb{N}^r$ for $r\geq 0$, and $|{\bf s}|=\sum_{i=1}^r s_i$, $l({\bf s})=r$.
\end{theorem}
\begin{proof}
The proof is by induction on $n.$ For $n=1$ we have
$S_1(\{2\}^a,c,\{2\}^b)=1,$
and the formula is true. For $n>1$ we proceed as follows.
If $c=2$ then we should prove that
\begin{equation}\label{S2m}
S_{n}(\{2\}^m)=2\sum_{k=1}^n\frac{(-1)^{k-1}\binom{n}{k}}{k^{2m}\binom{n+k}{k}}
\end{equation}
where $m=a+b+1$. First note that
\begin{equation*}
S_n(\{2\}^m)=\sum_{1\le k_1\le k_2\le \cdots\le k_m\le n}\frac{1}{k_1^2k_2^2\cdots k_m^2}
=\sum_{l=0}^m\frac{1}{n^{2(m-l)}}S_{n-1}(\{2\}^l).
\end{equation*}
Then by the induction hypothesis, we have that
\begin{align*}
S_n(\{2\}^m)&=2\sum_{l=0}^n\frac{1}{n^{2(m-l)}}\sum_{k=1}^{n-1}\frac{(-1)^{k-1}\binom{n-1}{k}}{k^{2l}\binom{n+k-1}{k}}
=\frac{2}{n^{2m}}\sum_{k=1}^{n-1}\frac{(-1)^{k-1}\binom{n-1}{k}}{\binom{n+k-1}{k}}\sum_{l=0}^m\frac{n^{2l}}{k^{2l}} \\[3pt]
&=\frac{2}{n^{2m+2}}\sum_{k=1}^n\frac{(-1)^{k-1}\binom{n}{k}(n^{2m+2}-k^{2m+2})}{k^{2m}\binom{n+k}{k}},\\
&=2\sum_{k=1}^n\frac{(-1)^{k-1}\binom{n}{k}}{k^{2m}\binom{n+k}{k}}-\frac{2}{n^{2m+2}}\sum_{k=1}^n\frac{(-1)^{k-1}k^2\binom{n}{k}}{\binom{n+k}{k}}
\end{align*}
and formula \eqref{S2m} easily follows by the equation \eqref{eq3.2.2}.

To prove \eqref{S2ac2b} for $c>2$ we note that for $n>1,$
\begin{equation*}
\begin{split}
S_n(\{2\}^a,c,\{2\}^b)&=\sum_{1\le k_1\le\cdots\le k_a\le k_{a+1}\le k_{a+2}\le\cdots\le k_{a+b+1}\le n}
\frac{1}{k_1^2\cdots k_a^2k_{a+1}^ck_{a+2}^2\cdots k_{a+b+1}^2} \\
&=\sum_{l=0}^b\frac{1}{n^{2(b-l)}} S_{n-1}(\{2\}^a,c,\{2\}^l)+\frac{1}{n^{2b+c}}S_n(\{2\}^a).
\end{split}
\end{equation*}
Then by the induction hypothesis  and formula (\ref{S2m}), we have
\begin{align*}
S_n(\{2\}^a,c,\{2\}^b)
&=2\sum_{l=0}^b\frac{1}{n^{2(b-l)}}\sum_{k=1}^{n-1}\frac{(-1)^{k-1}\binom{n-1}{k}}{k^{2(a+l)+c}\binom{n+k-1}{k}} \\
&+4\sum_{l=0}^b\frac{1}{n^{2(b-l)}}
\underset{i\geq 1,j\geq 2, |{\bf s}|\ge 0}{\sum_{i+j+|{\bf s}|=c}}
2^{l({\bf s})}
\sum_{k=1}^{n-1}\frac{H_{k-1}(2a+i,{\bf s})(-1)^{k-1}\binom{n-1}{k}}{k^{2l+j}\binom{n-1+k}{k}}
\\
&+\frac{2}{n^{2b+c}}\sum_{k=1}^{n}\frac{(-1)^{k-1}\binom{n}{k}}{k^{2a}\binom{n+k}{k}}.
\end{align*}
Changing the order of summation and summing the inner sum
\begin{equation*}
\frac{\binom{n-1}{k}}{\binom{n-1+k}{k}}\sum_{l=0}^b\frac{n^{2l}}{k^{2l}}
=\frac{n^{2b+2}-k^{2b+2}}{(n^2-k^2)k^{2b}}\frac{\binom{n-1}{k}}{\binom{n-1+k}{k}}
=\left(\frac{n^{2b}}{k^{2b}}-\frac{k^{2}}{n^{2}}\right)\frac{\binom{n}{k}}{\binom{n+k}{k}},
\end{equation*}
we obtain
\begin{align*}
S_n(\{2\}^a,c,\{2\}^b)&=2\sum_{k=1}^n \frac{(-1)^{k-1}\binom{n}{k}}{k^{2a+2b+c}\binom{n+k}{k}} \\
&+\;4\!\!\!\underset{i\geq 1,j\geq 2, |{\bf s}|\ge 0}{\sum_{i+j+|{\bf s}|=c}}
2^{l({\bf s})}
\sum_{k=1}^n\frac{H_{k-1}(2a+i,{\bf s})(-1)^{k-1}\binom{n}{k}}{k^{2b+j}\binom{n+k}{k}}\\
&+\frac{2}{n^{2b+c}}\sum_{k=1}^{n}\frac{(-1)^{k-1}\binom{n}{k}}{k^{2a}\binom{n+k}{k}}
-\frac{2}{n^{2b+2}}\sum_{k=1}^n \frac{(-1)^{k-1}\binom{n}{k}}{k^{2a+c-2}\binom{n+k}{k}}\\
&-\frac{4}{n^{2b+2}}\underset{i\geq 1,j\geq 2, |{\bf s}|\ge 0}{\sum_{i+j+|{\bf s}|=c}}
2^{l({\bf s})}
\sum_{k=1}^n\frac{H_{k-1}(2a+i,{\bf s})(-1)^{k-1}\binom{n}{k}}{k^{j-2}\binom{n+k}{k}}.
\end{align*}
The final result follows as soon as we show that
\begin{align*}
\frac{1}{n^{c-2}}\sum_{k=1}^{n}\frac{A_{n,k}^{(2)}}{k^{2a}}
&=\sum_{k=1}^n \frac{A_{n,k}^{(2)}}{k^{2a+c-2}}
+2\underset{i\geq 1,j\geq 2, |{\bf s}|\ge 0}{\sum_{i+j+|{\bf s}|=c}}
2^{l({\bf s})}
\sum_{k=1}^n\frac{H_{k-1}(2a+i,{\bf s})A_{n,k}^{(2)}}{k^{j-2}}
\end{align*}
where $A_{n,k}^{(2)}=(-1)^{k-1}\binom{n}{k}/\binom{n+k}{k}$. This identity holds by Lemma \ref{l3.2.2} because
\begin{equation*}
2\!\!\!\underset{i\geq 1,j\geq 2, |{\bf s}|\ge 0}{\sum_{i+j+|{\bf s}|=c}}
2^{l({\bf s})}
\sum_{k=1}^n\frac{H_{k-1}(2a+i,{\bf s})A_{n,k}^{(2)}}{k^{j-2}}
=\underset{j\ge 0, s_1>2a}{\sum_{j+|{\bf s}|=2a+c-2}}
2^{l({\bf s})}
\sum_{k=1}^n\frac{H_{k-1}({\bf s})A_{n,k}^{(2)}}{k^{j}}.
\end{equation*}
\end{proof}

\begin{corollary}
For a positive integer $n$ and non-negative integers $a, b,$
\begin{align}
S_n(\{2\}^a)&=2\sum_{k=1}^n\frac{(-1)^{k-1}\binom{n}{k}}{k^{2a}\binom{n+k}{k}},\nonumber\\[3pt]
S_n(\{2\}^a,3,\{2\}^b)&=2\sum_{k=1}^n\frac{(-1)^{k-1}\binom{n}{k}}{k^{2(a+b)+3}\binom{n+k}{k}}+
4\sum_{k=1}^n\frac{H_{k-1}(2a+1)(-1)^{k-1}\binom{n}{k}}{k^{2b+2}\binom{n+k}{k}}.
\label{S2a32b}
\end{align}
\end{corollary}

\begin{theorem} \label{t3.2.2}
For a positive integer $n$ and non-negative integers $a, b, c$,
\begin{equation}\label{S1ac1b}
S_n(\{1\}^a,c,\{1\}^b)=
\sum_{k=1}^n \frac{(-1)^{k-1}\binom{n}{k}}{k^{a+b+c}}
+\underset{i\geq 1,j\geq 1, |{\bf s}|\ge 0}{\sum_{i+j+|{\bf s}|=c}}\,
\sum_{k=1}^n\frac{H_{k-1}(a+i,{\bf s})(-1)^{k-1}\binom{n}{k}}{k^{b+j}}.
\end{equation}
\end{theorem}
\begin{proof}
Note that the case $c=1$ is well known (see \cite[Corollary 3]{Di:95}). It also follows from the polynomial identity
\begin{equation*}
\sum_{1\le k_1\le\ldots\le k_m\le n}\frac{(1+x)^{k_1}-1}{k_1\cdots k_m}=\sum_{k=1}^n\binom{n}{k}\frac{x^{k}}{k^m}
\end{equation*}
which appeared in \cite[Lemma 5.5]{TZ:10}. For $c>1$ the identity can be proved by induction on $n.$
By a similar argument as in the proof of Theorem \ref{T3.2.1}, we have
\begin{equation*}
\begin{split}
S_n(\{1\}^a, c, \{1\}^b)&=\sum_{l=0}^b\frac{1}{n^{b-l}}S_{n-1}(\{1\}^a,c,\{1\}^l)+\frac{1}{n^{b+c}}S_n(\{1\}^a) \\
&=\frac{1}{n^{b+c}}\sum_{k=1}^n\frac{(-1)^{k-1}\binom{n}{k}}{k^a}+
\sum_{l=0}^b\frac{1}{n^{b-l}}\sum_{k=1}^{n-1}\frac{(-1)^{k-1}\binom{n-1}{k}}{k^{a+c+l}} \\
&+\sum_{l=0}^b\frac{1}{n^{b-l}}
\!\!\!\underset{i\ge 1,  j\ge 1, |{\bf s}|\ge 0}{\sum_{i+j+|{\bf s}|=c}}\!\sum_{k=1}^{n-1}\frac{H_{k-1}(a+i,{\bf s})(-1)^{k-1}\binom{n-1}{k}}{k^{l+j}}.
\end{split}
\end{equation*}
Changing the order of summation and summing the inner sum
\begin{equation*}
\binom{n-1}{k}\sum_{l=0}^b\frac{n^l}{k^l}=\left(\frac{n^{b}}{k^b}-\frac{k}{n}\right)\binom{n}{k}
\end{equation*}
we obtain
\begin{equation*}
\begin{split}
S_n(\{1\}^a,c,\{1\}^b)&=\sum_{k=1}^n\frac{(-1)^{k-1}\binom{n}{k}}{k^{a+b+c}}+\underset{i\ge 1, j\ge 1, |{\bf s}|\ge 0}{\sum_{i+j+|{\bf s}|=c}}\,
\sum_{k=1}^n\frac{H_{k-1}(a+i, {\bf s})(-1)^{k-1}\binom{n}{k}}{k^{b+j}} \\
&+\frac{1}{n^{b+c}}\sum_{k=1}^n\frac{(-1)^{k-1}\binom{n}{k}}{k^a}
-\frac{1}{n^{b+1}}\sum_{k=1}^n\frac{(-1)^{k-1}\binom{n}{k}}{k^{a+c-1}} \\
&-\frac{1}{n^{b+1}}\underset{i\ge 1, j\ge 1, |{\bf s}|\ge 0}{\sum_{i+j+|{\bf s}|=c}}\,\sum_{k=1}^n\frac{H_{k-1}(a+i,{\bf s})(-1)^{k-1}\binom{n}{k}}{k^{j-1}}.
\end{split}
\end{equation*}
Now the result easily follows from Lemma \ref{l3.2.2} with
$A_{n,k}^{(1)}=(-1)^{k-1}\binom{n}{k}$ and the equality
\begin{equation*}
\underset{i\ge 1, j\ge 1, |{\bf s}|\ge 0}{\sum_{i+j+|{\bf s}|=c}}\,\sum_{k=1}^n\frac{H_{k-1}(a+i,{\bf s})(-1)^{k-1}\binom{n}{k}}{k^{j-1}}
=\underset{j\ge 0, s_1>a}{\sum_{j+|{\bf s}|=a+c-1}}\,\sum_{k=1}^n\frac{H_{k-1}({\bf s})(-1)^{k-1}\binom{n}{k}}{k^j}.
\end{equation*}
\end{proof}

\begin{theorem} \label{T3.2.2}
Let $a,b$ be integers satisfying $a\ge 1,$ $b\ge 0.$
Then for any positive integer $n,$
\begin{align}\label{eq34}
S_n(1,\{2\}^b)&=2\sum_{k=1}^n\frac{\binom{n}{k}}{k^{2b+1}\binom{n+k}{k}},\\[3pt]
S_n(\{2\}^a,1,\{2\}^b)&=2\sum_{k=1}^n\frac{(-1)^{k-1}\binom{n}{k}}{k^{2(a+b)+1}\binom{n+k}{k}}-
4\sum_{k=1}^n\frac{H_{k-1}(-2a)\binom{n}{k}}{k^{2b+1}\binom{n+k}{k}}. \label{eq35}
\end{align}
\end{theorem}
\begin{proof}
To prove the first identity we proceed by induction on $b.$ For $b=0$ its validity follows from (\ref{eq3.2.1d}).
Assume the formula holds for $b>0.$ Then we easily obtain from~(\ref{eq3.2.1c}),
\begin{equation*}
\begin{split}
S_n(1,\{2\}^{b+1})&=\sum_{k=1}^n\frac{S_k(1,\{2\}^b)}{k^2}=2\sum_{k=1}^n\frac{1}{k^2}\sum_{l=1}^k
\frac{\binom{k}{l}}{l^{2m+1}\binom{k+l}{l}} \\
&=2\sum_{l=1}^n\frac{1}{l^{2m+1}}\sum_{k=l}^n\frac{\binom{k}{l}}{k^2\binom{k+l}{l}}
=2\sum_{l=1}^n\frac{\binom{n}{l}}{l^{2m+3}\binom{n+l}{l}},
\end{split}
\end{equation*}
as required. To prove the second identity of our theorem we proceed by induction on $n.$ Obviously, it is valid for $n=1.$
For $n>1$ we use the equality
\begin{equation*}
S_n(\{2\}^a,1,\{2\}^b)=\sum_{l=0}^b\frac{1}{n^{2(b-l)}}S_{n-1}(\{2\}^a,1,\{2\}^l)
+\frac{1}{n^{2b+1}}S_n(\{2\}^a)
\end{equation*}
and apply the same arguments as in the proof of Theorem \ref{T3.2.1}. Then with the help of the formula (\ref{eq3.2.1b})
we may easily deduce the result.
\end{proof}
\begin{remark}
Note that the identities for the finite multiple harmonic sums found in Theorems \ref{T3.2.1}, \ref{T3.2.2}
can be used for evaluating multiple zeta-star values. For example, letting $n$ tend to infinity in (\ref{S2a32b}), (\ref{eq34}), (\ref{eq35})
we get expressions for $\zeta^*(\{2\}^a,3,\{2\}^b),$ $\zeta^*(\{2\}^a,1,\{2\}^b)$ in terms of alternating Euler sums:
\begin{align} \label{eq44}
\zeta^*(\{2\}^a,3,\{2\}^b)&=2\overline{\zeta}(2a+2b+3)+4\sum_{k=1}^{\infty}\frac{(-1)^{k-1}H_{k-1}(2a+1)}{k^{2b+2}}, \\
 \label{eq45}
\zeta^*(\{2\}^a,1,\{2\}^b)&=2\overline{\zeta}(2a+2b+1)-4\sum_{k=1}^{\infty}\frac{H_{k-1}(-2a)}{k^{2b+1}}, \quad a, b\ge 1, \\
\zeta^*(1,\{2\}^b)&=2\zeta(2b+1), \qquad b\ge 1, \label{eq45.5} \end{align}
where
\begin{equation*}
\overline{\zeta}(s):=\sum_{k=1}^{\infty}\frac{(-1)^{k-1}}{k^s}=(1-2^{1-s})\zeta(s),
\end{equation*}
with $\overline{\zeta}(1)=\log 2,$ is the alternating zeta function.
Formula (\ref{eq45.5}) was proved previously by several authors (see \cite[p.~292, Ex.~b]{OW:06}, \cite{Zl:05}, \cite{Va:96}).
Evaluations for length $2$ Euler sums of odd weight appearing on the right-hand sides of (\ref{eq44}) and (\ref{eq45}) in terms of zeta values are well known (see, for example, \cite[Theorem 7.2]{FS:98}).
Therefore by \cite[Theorem 7.2]{FS:98}, we obtain
\begin{equation*}
\zeta^*(\{2\}^a,3,\{2\}^b)=4\sum_{r=1}^{K}\left(\binom{2r}{2b+1}\Bigl(1-\frac{1}{2^{2r}}\Bigr)+\delta_{r,a}-
\binom{2r}{2a}\right)\zeta(2r+1)\overline{\zeta}(2K-2r),
\end{equation*}
where $\overline{\zeta}(0):=1$, $K=a+b+1$,
which gives another proof of Zagier's formula (\ref{zzzz}) for the zeta-star value $\zeta^*(\{2\}^a,3,\{2\}^b),$
and
\begin{equation} \label{121}
\zeta^*(\{2\}^a,1,\{2\}^b)=4\sum_{r=1}^{a+b}\left(\binom{2r}{2b}-\binom{2r}{2a-1}\Bigl(1-\frac{1}{2^{2r}}\Bigr)
\right)\zeta(2r+1)\overline{\zeta}(2a+2b-2r)
\end{equation}
for $a, b\ge 1.$ Note that formula (\ref{121}) was also proved in \cite[Theorem 1.6]{TY:12} by another method.
\end{remark}

\section{Auxiliary  results on congruences}

In this section we collect several congruences that will be required later in this paper.

\begin{enumerate}

\item[(i)] (\cite[Theorem 1.6]{Zhao:06}) for positive integers $a, r$ and for any prime $p>ar+2,$
\begin{align*}
H_{p-1}(\{a\}^r)\equiv
\begin{cases}
(-1)^r\frac{a(ar+1)}{2(ar+2)}\,p^2\,B_{p-ar-2} \pmod{p^3}
 &\mbox{if  $ar$ is odd,}\\
 (-1)^{r-1}\frac{a}{ar+1}\,p\,B_{p-ar-1}
  \pmod{p^2} &\mbox{if  $ar$ is even;}
\end{cases}
\end{align*}

\item[(ii)] (\cite[Theorems 3.1, 3.2]{Zhao:06}) for positive integers $a_1, a_2$ and for any prime $p\ge a_1+a_2,$
\begin{equation*}
H_{p-1}(a_1,a_2)\equiv\frac{(-1)^{a_2}}{a_1+a_2}\binom{a_1+a_2}{a_1}B_{p-a_1-a_2} \pmod{p},
\end{equation*}
moreover, if $a_1+a_2$ is even, then for any prime $p>a_1+a_2+1,$
\begin{equation*}
\begin{split}
H_{p-1}(a_1,a_2)&\equiv p\left[(-1)^{a_1}a_2\binom{a_1+a_2+1}{a_1}-(-1)^{a_1}a_1\binom{a_1+a_2+1}{a_2}-a_1-a_2\right] \\[3pt]
&\times
\frac{B_{p-a_1-a_2-1}}{2(a_1+a_2+1)} \pmod{p^2};
\end{split}
\end{equation*}

\item[(iii)] (\cite[Theorem 3.5]{Zhao:06} if $a_1, a_2, a_3\in {\mathbb N}$  and $w:=a_1+a_2+a_3$ is odd, then for any prime $p>w,$
we have
\begin{equation*}
H_{p-1}(a_1,a_2,a_3)\equiv \left[(-1)^{a_1}\binom{w}{a_1}-(-1)^{a_3}\binom{w}{a_3}\right]\frac{B_{p-w}}{2w} \pmod{p};
\end{equation*}

\item[(iv)] (\cite[Theorem 5.2]{Sunzh:00}) for a positive integer $a$ and for any prime $p\ge a+2,$ we have
\begin{align*}
H_{\frac{p-1}{2}}(a)\equiv
\begin{cases}
-2q_p(2)+p q^2_p(2)-p^2\left(\frac{2}{3}\,q^3_p(2)+\frac{7}{12}\,B_{p-3}\right)  \!\!\!\!\pmod{p^3} &\mbox{\!\!if $a=1$,}\\[6pt]
\frac{a(2^{a+1}-1)}{2(a+1)}\,p\,B_{p-a-1}  \pmod{p^2} &\mbox{\!\!if $a$ is even,}\\[6pt]
-\frac{2^a-2}{a}\,B_{p-a}  \pmod{p} &\mbox{\!\!if $a>1$ is odd,}
\end{cases}
\end{align*}
where $q_p(2)=(2^{p-1}-1)/p$ is the so-called {\sl Fermat quotient};

\vspace{0.3cm}

\item[(v)] (\cite[Lemma 1]{HHT:11}) if $a, b$ are positive integers and $a+b$ is odd, then for any prime $p>a+b,$
\begin{equation*}
H_{\frac{p-1}{2}}(a,b)\equiv \frac{B_{p-a-b}}{2(a+b)}\left((-1)^b\binom{a+b}{a}+2^{a+b}-2\right) \pmod{p};
\end{equation*}

\item[(vi)] by (i) and (iv), for any positive integer $a$ and any prime $p\ge a+2,$ we have
\begin{align*}
H_{p-1}(-a)&=2^{1-a}H_{\frac{p-1}{2}}(a)-H_{p-1}(a)\\
&\equiv
\begin{cases}
-2q_p(2)+p q^2_p(2)-p^2\left(\frac{2}{3}\,q^3_p(2)+\frac{1}{4}\,B_{p-3}\right) \!\!\!\!\! \pmod{p^3} &\mbox{\!\!\!if $a=1$,}\\[6pt]
\frac{a(1-2^{-a})}{a+1}\,p\,B_{p-a-1}  \pmod{p^2} &\mbox{\!\!\!if $a$ is even,}\\[6pt]
-\frac{2(1-2^{1-a})}{a}\,B_{p-a}\pmod{p} &\mbox{\!\!\!if $a>1$ is odd;}
\end{cases}
\end{align*}

\item[(vii)] (\cite[Theorem 3.1]{TZ:10}) for positive integers $a, b$ of distinct parity and a prime $p\ge a+b+1,$
\begin{equation*}
H_{p-1}(-a,b)\equiv H_{p-1}(a, -b)\equiv \frac{1-2^{1-a-b}}{a+b}\, B_{p-a-b} \pmod{p},
\end{equation*}
\begin{equation*}
H_{p-1}(-a,-b)\equiv \frac{2^{1-a-b}-1}{a+b}\,(-1)^b\binom{a+b}{b}B_{p-a-b} \pmod{p};
\end{equation*}

\item[(viii)] (\cite[Theorem 4.1]{TZ:10})
if $a, b, c\in {\mathbb N}$  and $w:=a+b+c$ is odd, then for any prime $p>w,$
\begin{equation*}
\begin{split}
\qquad\qquad 2H_{p-1}(a,-b,-c)&\equiv H_{p-1}(c+b,a)+H_{p-1}(-c,-b-a) \\
&\qquad\qquad\quad -H_{p-1}(-c)H_{p-1}(-b,a) \pmod{p}, \\[3pt]
\qquad\qquad 2H_{p-1}(-a,b,-c)&\equiv H_{p-1}(-c)H_{p-1}(b,-a)-H_{p-1}(-c,b)H_{p-1}(-a) \\
&\quad\quad +H_{p-1}(-c-b,-a)+H_{p-1}(-c,-b-a) \pmod{p}.
\end{split}
\end{equation*}
\end{enumerate}

\vspace{0.1cm}

In \cite[Lemma 6.2]{CD:06}  it was proved that for positive integers $a,b$ and a prime $p>2a+2b+1,$
\begin{equation*}
H_{p-1}(-2a, -2b)\equiv 0 \pmod{p}.
\end{equation*}
However, we will need a stronger version of this congruence.
\begin{lemma} \label{l0}
Let $a,b$ be positive integers and a prime $p>2a+2b+1.$ Then the following congruence holds modulo $p^2:$
\begin{equation*}
H_{p-1}(-2a,-2b)\equiv\left(\frac{(a-b)(1-2^{-2a-2b})}{(2a+1)(2b+1)}\binom{2a+2b}{2a}-\frac{a+b}{2a+2b+1}\right)pB_{p-2a-2b-1}.
\end{equation*}
\end{lemma}
\begin{proof}
Clearly (see \cite[Lemma 7.1]{TZ:10}),
\begin{equation*}
\begin{split}
H_{p-1}(-2a,-2b)&=\sum_{1\le j<k\le p-1}\frac{(-1)^{j+k}}{j^{2a}k^{2b}}=\sum_{1\le k<j\le p-1}\frac{(-1)^{p-j+p-k}}{(p-j)^{2a}(p-k)^{2b}}\\[2pt]
&\equiv \sum_{1\le k<j\le p-1}\frac{(-1)^{k+j}}{k^{2b}j^{2a}}\left(1+\frac{2bp}{k}\right)\left(1+\frac{2ap}{j}\right)
\equiv
H_{p-1}(-2b,-2a) \\[2pt]
&\,+2bpH_{p-1}(-2b-1,-2a)+2apH_{p-1}(-2b,-2a-1) \pmod{p^2}.
\end{split}
\end{equation*}
Therefore by (vii), we obtain
\begin{equation} \label{eq32}
\begin{split}
 H_{p-1}&(-2a,-2b)\equiv H_{p-1}(-2b,-2a) \\[3pt]
&\qquad\quad+\frac{2p(a-b)(1-2^{-2a-2b})}{(2a+1)(2b+1)}\binom{2a+2b}{2a}B_{p-2a-2b-1} \pmod{p^2}.
\end{split}
\end{equation}
On the other hand, we have
\begin{equation*}
\begin{split}
H_{p-1}(-2a,-2b)&+H_{p-1}(-2b,-2a) \\
&\equiv H_{p-1}(-2a)H_{p-1}(-2b)-H_{p-1}(2a+2b) \pmod{p^2}
\end{split}
\end{equation*}
and therefore by (vi) and (i),
\begin{equation} \label{eq33}
\begin{split}
H_{p-1}(-2a,-2b)&+H_{p-1}(-2b,-2a)\equiv -H_{p-1}(2a+2b) \\
&\equiv
-\frac{(2a+2b)p}{2a+2b+1}\,B_{p-2a-2b-1} \pmod{p^2}.
\end{split}
\end{equation}
Now from (\ref{eq32}) and (\ref{eq33}) we obtain the required congruence.
\end{proof}

\begin{lemma} \label{l02}
For positive integers $a, b$  and a prime $p>2a+2b+1,$ the following congruence holds modulo $p:$
\begin{equation*}
H_{\frac{p-1}{2}}(-2a,-2b-1)\equiv \frac{2^{2a+2b}-1}{2a+2b+1}\left(\frac{1}{2^{2a+2b+1}}\binom{2a+2b+1}{2a}+1\right)B_{p-2a-2b-1}.
\end{equation*}
\end{lemma}
\begin{proof}
It is easily seen that
\begin{equation} \label{eq37}
\begin{split}
&H_{p-1}(-2a,-2b-1)-H_{\frac{p-1}{2}}(-2a,-2b-1)
=\sum_{k=\frac{p+1}{2}}^{p-1}\frac{(-1)^kH_{k-1}(-2a)}{k^{2b+1}} \\
&=\sum_{k=1}^{(p-1)/2}\frac{(-1)^{p-k}H_{p-k-1}(-2a)}{(p-k)^{2b+1}}
\equiv \sum_{k=1}^{(p-1)/2}\frac{(-1)^kH_{p-k-1}(-2a)}{k^{2b+1}} \pmod{p}.
\end{split}
\end{equation}
For the harmonic number $H_{p-k-1}(-2a),$ by (vi) we have
\begin{equation*}
\begin{split}
H_{p-k-1}(-2a)
&=\!\!\sum_{j=1}^{p-k-1}\frac{(-1)^{j}}{j^{2a}}
=\!\!\sum_{j=k+1}^{p-1}\frac{(-1)^{p-j}}{(p-j)^{2a}}
\equiv -\!\!\!\sum_{j=k+1}^{p-1}\frac{(-1)^j}{j^{2a}}\\
&=-(H_{p-1}(-2a)-H_k(-2a))\equiv H_k(-2a) \pmod{p}.
\end{split}
\end{equation*}
Substituting the above congruence in (\ref{eq37}) we obtain
\begin{equation*}
\begin{split}
H_{p-1}(-2a,-2b-1)&-H_{\frac{p-1}{2}}(-2a,-2b-1)\equiv \!\!\!\sum_{k=1}^{(p-1)/2}\frac{(-1)^kH_k(-2a)}{k^{2b+1}}\\
&=H_{\frac{p-1}{2}}(-2a,-2b-1)
+H_{\frac{p-1}{2}}(2a+2b+1) \pmod{p},
\end{split}
\end{equation*}
which by (iv) and (vii) implies the required result.
\end{proof}

\section{New congruences for multiple harmonic sums}
\label{S4}

In this section we state and prove our main results on multiple harmonic sums.

\begin{theorem} \label{T3.3.1}
Let $a,b$ be non-negative integers and $p$ be a prime such that $p>2a+2b+3.$ Then
\begin{equation*}
\begin{split}
S_{p-1}(\{2\}^a,3,\{2\}^b)&\equiv \frac{b-a}{(a+1)(b+1)}\binom{2a+2b+2}{2a+1} B_{p-2a-2b-3} \pmod{p},\\[2pt]
H_{p-1}(\{2\}^a,3,\{2\}^b) &\equiv \frac{(-1)^{a+b}(a-b)}{(a+1)(b+1)}\binom{2a+2b+2}{2a+1} B_{p-2a-2b-3} \pmod{p}.
\end{split}
\end{equation*}
\end{theorem}
\begin{proof}
From Theorem \ref{T3.2.1} we have
\begin{equation*}
S_{p-1}(\{2\}^a,3,\{2\}^b)=2\sum_{k=1}^{p-1}\frac{(-1)^{k-1}\binom{p-1}{k}}{k^{2a+2b+3}\binom{p-1+k}{k}}
+4\sum_{k=1}^{p-1}\frac{H_{k-1}(2a+1)(-1)^{k-1}\binom{p-1}{k}}{k^{2b+2}\binom{p-1+k}{k}}.
\end{equation*}
Since
\begin{equation} \label{eq26}
\begin{split}
\frac{\binom{p-1}{k}}{\binom{p-1+k}{k}}&=\frac{(p-1)(p-2)\cdots(p-k)}{p(p+1)\cdots(p+k-1)}
=\frac{(-1)^k k}{p}\prod_{j=1}^k\left(1-\frac{p}{j}\right)\prod_{j=1}^{k-1}\left(1+\frac{p}{j}\right)^{-1} \\
&\equiv \frac{(-1)^k k}{p}(1-pH_k(1))(1-pH_{k-1}(1)) \\
&\equiv \frac{(-1)^k k}{p}\left(1-2pH_{k-1}(1)-\frac{p}{k}\right)\pmod{p},
\end{split}
\end{equation}
we obtain
\begin{equation*}
\begin{split}
S_{p-1}(\{2\}^a,3,\{2\}^b)&\equiv -\frac{2}{p}\sum_{k=1}^{p-1}\frac{1}{k^{2a+2b+2}}\left(1-2pH_{k-1}(1)-\frac{p}{k}\right) \\
&-\frac{4}{p}\sum_{k=1}^{p-1}\frac{H_{k-1}(2a+1)}{k^{2b+1}}\left(1-2pH_{k-1}(1)-\frac{p}{k}\right) \\
&=-\frac{2}{p}H_{p-1}(2a+2b+2)+4H_{p-1}(1,2a+2b+2) \\
&-\frac{4}{p}H_{p-1}(2a+1,2b+1)
+8H_{p-1}(2a+1,1,2b+1) \\
&+
8H_{p-1}(1,2a+1,2b+1)
+8H_{p-1}(2a+2,2b+1) \\
&+4H_{p-1}(2a+1,2b+2) \pmod{p}.
\end{split}
\end{equation*}
Now by (i)--(iii), we easily obtain the required congruence for $S_{p-1}(\{2\}^a,3,\{2\}^b).$
To prove the corresponding congruence for the $H$-sum, we apply the following identity which relates
$S$-version and $H$-version multiple harmonic sums (see \cite[Theorem~6.8]{Ho:03}):
\begin{equation} \label{SH}
(-1)^{l({\bf s})}S_{n}({\overline{{\bf s}}})=
\sum_{\bigsqcup_{i=1}^l{\bf s}_i={\bf s}}(-1)^l\prod_{i=1}^lH_{n}({\bf s}_i),
\end{equation}
where $\bigsqcup_{i=1}^l{\bf s}_i$ is the concatenation of ${\bf s}_1$ to ${\bf s}_l$ and the vector $\overline{\bf s}=(s_r, \ldots, s_1)$
is  obtained from ${\bf s}=(s_1, \ldots, s_r)$ by reversing its coordinates.
We substitute ${\bf s}=(\{2\}^b,3,\{2\}^a)$ and $n=p-1$ in (\ref{SH}) and note that if $l\ge 2,$ then one of ${\bf s}_i=(\{2\}^c)$
for some positive $c$ and therefore $H_{p-1}({\bf s}_i)\equiv 0 \pmod{p}.$ Thus
\begin{equation} \label{eq27}
(-1)^{a+b+1} S_{p-1}(\{2\}^a,3,\{2\}^b)\equiv -H_{p-1}(\{2\}^b,3,\{2\}^a) \pmod{p}.
\end{equation}
On the other hand, we have (see \cite[Theorem 4.5]{Hof:07})
\begin{equation} \label{eq28}
H_{p-1}(\{2\}^a,3,\{2\}^b)\equiv -H_{p-1}(\{2\}^b,3,\{2\}^a) \pmod{p}.
\end{equation}
Now from (\ref{eq27}) and (\ref{eq28}) we get the required congruence and the proof is complete.
\end{proof}

\begin{theorem} \label{T3.3.2}
Let $a, b$ be  non-negative integers and a prime  $p>2a+2b+1.$ Then
\begin{equation*}
\begin{split}
S_{p-1}(\{2\}^a,1,\{2\}^b) &\equiv \frac{4(b-a)(1-4^{-a-b})}{(2a+1)(2b+1)}\binom{2a+2b}{2a}B_{p-2a-2b-1} \pmod{p},\\[2pt]
H_{p-1}(\{2\}^a,1,\{2\}^b) &\equiv \frac{4(-1)^{a+b}(a-b)(1-4^{-a-b})}{(2a+1)(2b+1)}\binom{2a+2b}{2a}B_{p-2a-2b-1} \!\!\!\pmod{p}.
\end{split}
\end{equation*}
\end{theorem}
\begin{proof}
We begin by considering the case $a\ge 1.$ Then from identity (\ref{eq35}) with $n=p-1$ and
congruence (\ref{eq26}) we have
\begin{equation*}
\begin{split}
S_{p-1}(\{2\}^a,1,\{2\}^b)&\equiv
-\frac{2}{p}\sum_{k=1}^{p-1}\frac{1}{k^{2a+2b}}\left(1-2pH_{k-1}(1)-\frac{p}{k}\right) \\
&-\frac{4}{p}\sum_{k=1}^{p-1}\frac{H_{k-1}(-2a)(-1)^k}{k^{2b}}\left(1-2pH_{k-1}(1)-\frac{p}{k}\right)
\pmod{p}.
\end{split}
\end{equation*}
By the obvious equality
\begin{equation}
H_n(a)H_n(b)=H_n(a,b)+H_n(b,a)+H_n\bigl({\textup{sgn}}(ab)(|a|+|b|)\bigr) \quad a,b\in {\mathbb Z}\setminus\{0\},
\label{eq36}
\end{equation}
we obtain
\begin{equation*}
\begin{split}
S_{p-1}(\{2\}^a,1,\{2\}^b)\equiv &-\frac{2}{p}H_{p-1}(2a+2b)+4H_{p-1}(1,2a+2b)-\frac{4}{p}H_{p-1}(-2a,-2b) \\
&+8H_{p-1}(-2a,1,-2b)+8H_{p-1}(1,-2a,-2b) \\
&+8H_{p-1}(-2a-1,-2b)+4H_{p-1}(-2a,-2b-1) \pmod{p}.
\end{split}
\end{equation*}
Applying (viii), (\ref{eq36}), (vi) and simplifying we deduce that
\begin{equation*}
S_{p-1}(\{2\}^a,1,\{2\}^b)\equiv -\frac{2}{p}H_{p-1}(2a+2b)-\frac{4}{p}H_{p-1}(-2a,-2b) \pmod{p}.
\end{equation*}
Now the required congruence follows from (i) and Lemma \ref{l0}. The case $a=0$ is handled in the same way with the help
of identity (\ref{eq34}) and congruences (vi), (vii). The corresponding congruence for $H_{p-1}(\{2\}^a,1,\{2\}^b)$ easily follows
from relation~(\ref{SH}).
\end{proof}

\begin{corollary} \label{c3}
Let $p$ be a prime greater than $7.$ Then
\begin{equation} \label{eq40}
S_{p-1}(1,1,1,4)\equiv S_{p-1}(1,2,2,2)\equiv \frac{27}{16}\,B_{p-7} \pmod{p},
\end{equation}
\begin{equation} \label{eq41}
S_{p-1}(1,1,1,6)\equiv \frac{1}{54}B_{p-3}^3+\frac{1889}{648}\,B_{p-9} \pmod{p}.
\end{equation}
\end{corollary}
\begin{proof}
The first congruence in (\ref{eq40}) follows from \cite[Theorem 7.3]{Hof:07} and the second one is a consequence of Theorem \ref{T3.3.2}
with $a=0,$ $b=3.$ To prove congruence (\ref{eq41}) we set $a=0,$ $b=4$ in Theorem \ref{T3.3.2} to obtain
\begin{equation} \label{eq42}
S_{p-1}(1,\{2\}^4)\equiv \frac{85}{48}\,B_{p-9} \pmod{p}.
\end{equation}
On the other hand, using \cite[Theorem 6.4]{Hof:07} we can rewrite the length 5 sum $S_{p-1}(1,\{2\}^4)$ in terms of length 4 sums
and $S_{p-1}(1,8)$ as
\begin{equation*}
\begin{split}
2S_{p-1}(1,\{2\}^4)\equiv &S_{p-1}(1,2,2,4)+S_{p-1}(1,2,4,2)+S_{p-1}(1,4,2,2) \\
&+S_{p-1}(3,2,2,2)-\frac{255}{18}S_{p-1}(1,8) \pmod{p}.
\end{split}
\end{equation*}
Now applying \cite[Theorem 7.5]{Hof:07} and expressing all length 4 sums in the congruence above in terms of the three quantities
$S_{p-1}(1,1,1,6),$ $S_{p-1}(1,8)\equiv  B_{p-9},$ and $S_{p-1}(1,2)S_{p-1}(1,1,4)\equiv -\frac{1}{6}B_{p-3}^3 \pmod{p}$ we obtain
\begin{equation} \label{eq43}
2S_{p-1}(1,\!\{2\}^4)\equiv 3S_{p-1}(1,1,1,6)+\frac{1}{3}S_{p-1}(1,2)S_{p-1}(1,1,4)-\frac{281}{54}S_{p-1}(1,8) \!\!\!\!\!\pmod{p}.
\end{equation}
Comparing congruences (\ref{eq42}) and (\ref{eq43}) we conclude that
\begin{equation*}
\begin{split}
S_{p-1}(1,1,1,6)&\equiv -\frac{1}{9}S_{p-1}(1,2)S_{p-1}(1,1,4)+\frac{1889}{648}S_{p-1}(1,8) \\[3pt]
&\equiv
\frac{1}{54}\,B_{p-3}^3+\frac{1889}{648}\,B_{p-9} \pmod{p}.
\end{split}
\end{equation*}
\end{proof}
From Corollary \ref{c3}, \cite[Theorems 7.3, 7.5]{Hof:07}, and \cite[Section 2]{Zhao:11} we obtain
the following description of multiple harmonic sums of weight 7 and 9.

\begin{corollary} \label{c4.2}
Let $p$ be a prime greater than $7.$ Then

\begin{itemize}

\item[(i)] all multiple harmonic sums $S_{p-1}({\bf s}),$ $H_{p-1}({\bf s}),$ of weight $|{\bf s}|=7$
belong to  ${\mathbb Q}(2) B_{p-7}$ modulo $p$;

\item[(ii)]  all multiple harmonic sums $S_{p-1}({\bf s}),$ $H_{p-1}({\bf s})$ of weight $|{\bf s}|=9$
can be written modulo $p$ as linear combinations of the two quantities $B_{p-9}$ and $B_{p-3}^3$ with
coefficients in ${\mathbb Q}(3).$
\end{itemize}
\end{corollary}

In \cite[Theorem 3.16]{Zhao:06}, Zhao proved that
\begin{equation*}
S_{p-1}(\{1\}^a,2,\{1\}^b)\equiv H_{p-1}(\{1\}^a,2,\{1\}^b)\equiv 0 \pmod{p}
\end{equation*}
if $a+b$ is even and a prime $p>a+b+3.$ In the next theorem we prove a deeper result on the above sums.

\begin{theorem} \label{T3.2}
For non-negative integers $a,b$  and  a prime  $p>a+b+3,$
\begin{align*}
S_{p-1}(\{1\}^a,2,\{1\}^b)\equiv
\begin{cases}
\frac{(-1)^b}{a+b+2}\,\binom{a+b+2}{a+1}B_{p-a-b-2} \quad\quad\,\,
 \pmod{p} &\mbox{\!\!if $a+b$ is odd,}\\[6pt]
\frac{pB_{p-a-b-3}}{2(a+b+3)}\,\left(1+(-1)^a\binom{a+b+3}{a+2}\right)  \!\!\!\!\pmod{p^2} &\mbox{\!\!if $a+b$ is even;}
\end{cases}
\end{align*}
\begin{align*}
H_{p-1}(\{1\}^a,2,\{1\}^b)\equiv
\begin{cases}
\frac{(-1)^b}{a+b+2}\,\binom{a+b+2}{a+1}B_{p-a-b-2} \quad\quad\,\,
 \pmod{p} &\mbox{\!\!if $a+b$ is odd,}\\[6pt]
\frac{pB_{p-a-b-3}}{2(a+b+3)}\,\left(1+(-1)^a\binom{a+b+3}{b+2}\right)  \!\!\!\!\!\pmod{p^2} &\mbox{\!\!if $a+b$ is even.}
\end{cases}
\end{align*}
\end{theorem}
\begin{proof}
From Theorem \ref{t3.2.2} for $c=2$ we have
\begin{equation*}
S_n(\{1\}^a,2,\{1\}^b)=\sum_{k=1}^n\frac{(-1)^{k-1}\binom{n}{k}}{k^{a+b+2}}+\sum_{k=1}^n\frac{H_{k-1}(a+1)(-1)^{k-1}\binom{n}{k}}{k^{b+1}}.
\end{equation*}
Setting $n=p-1$ and observing that
\begin{equation*}
\binom{p-1}{k}=\frac{(p-1)(p-2)\cdots(p-k)}{k!}\equiv (-1)^k(1-pH_k(1)) \pmod{p^2}
\end{equation*}
we have
\begin{equation} \label{eq29}
\begin{split}
S_{p-1}&(\{1\}^a,2,\{1\}^b)\equiv -\sum_{k=1}^{p-1}\frac{1-pH_{k}(1)}{k^{a+b+2}}-\sum_{k=1}^{p-1}\frac{H_{k-1}(a+1)(1-pH_k(1))}{k^{b+1}}\\
&=-H_{p-1}(a+b+2)+pH_{p-1}(1,a+b+2)+pH_{p-1}(a+b+3) \\[3pt]
&-H_{p-1}(a+1,b+1)+pH_{p-1}(a+1,b+2)+pH_{p-1}(a+1,1,b+1) \\[2pt]
&+pH_{p-1}(1,a+1,b+1)
+pH_{p-1}(a+2,b+1) \pmod{p^2}.
\end{split}
\end{equation}
If $a+b$ is odd, then by (ii) we obtain
\begin{equation*}
S_{p-1}(\{1\}^a,2,\{1\}^b)\!\equiv\! -H_{p-1}(a+1,b+1)\!\equiv\!\frac{(-1)^b}{a+b+2}\binom{a+b+2}{a+1} B_{p-a-b-2} \!\!\! \pmod{p}.
\end{equation*}
If $a+b$ is even, then from (\ref{eq29}), (i)--(iii) after simplifying we get the required congruence.

Considering equality (\ref{SH}) with $n=p-1,$  ${\bf s}=(\{1\}^a,2,\{1\}^b)$ modulo $p$ we find
\begin{equation*}
(-1)^{a+b+1}S_{p-1}(\{1\}^b,2,\{1\}^a)\equiv -H_{p-1}(\{1\}^a,2,\{1\}^b) \pmod{p}
\end{equation*}
and therefore
\begin{align} \label{eq30}
H_{p-1}(\{1\}^a,2,\{1\}^b)\equiv
\begin{cases}
-S_{p-1}(\{1\}^b,2,\{1\}^a)
 \pmod{p} &\mbox{if $a+b$ is odd,}\\[4pt]
\quad\; 0\qquad \qquad \qquad \quad \pmod{p} &\mbox{if $a+b$ is even.}
\end{cases}
\end{align}
Similarly, from identity (\ref{SH}) considered modulo $p^2$ we obtain
\begin{equation} \label{eq31}
\begin{split}
(-1)^{a+b}S_{p-1}&(\{1\}^b,2,\{1\}^a) \\
&\equiv H_{p-1}(\{1\}^a,2,\{1\}^b)-
\!\!\!\!\!\!\!\!\underset{(\{1\}^a,2,\{1\}^b)}{\prod_{{\bf s}_1\bigsqcup{\bf s}_2=}}\!\!\!\!\!\!\!
H_{p-1}({\bf s}_1)H_{p-1}({\bf s}_2) \pmod{p^2}.
\end{split}
\end{equation}
It is clear that the product $H_{p-1}({\bf s}_1)H_{p-1}({\bf s}_2)$ has one of the forms
\begin{equation*}
H_{p-1}(\{1\}^m)H_{p-1}(\{1\}^{a-m},2,\{1\}^b), \qquad 1\le m\le a,
\end{equation*}
or
\begin{equation*}
H_{p-1}(\{1\}^a,2,\{1\}^l)H_{p-1}(\{1\}^{b-l}), \qquad 0\le l\le b.
\end{equation*}
If $a+b$ is even, then by (i) and (\ref{eq30}), it easily follows that both products are congruent to zero modulo $p^2.$
Thus in this case from (\ref{eq31}) we obtain
\begin{equation*}
S_{p-1}(\{1\}^b,2,\{1\}^a)\equiv H_{p-1}(\{1\}^a,2,\{1\}^b) \pmod{p^2}
\end{equation*}
and the proof is complete.
\end{proof}

\begin{remark}
It is worth mentioning that there is a notion of duality for multiple harmonic sums:
given ${\bf s}=(s_1,s_2,\dots,s_r)$ we define the power set
$P({\bf s})$ to be the
partial sum sequence $(s_1, s_1+s_2,\dots,s_1+\dots + s_{r-1})$ as a subset of
$\{1, 2, \dots, |{\bf s}|-1\}$.
Then ${\bf s}^{*}$ is the composition of weight $|{\bf s}|$ corresponding
to the complement subset of $P({\bf s})$ in $\{1, 2, \dots, |{\bf s}|-1\}$, namely,
\begin{equation*}
{\bf s}^{*}=P^{-1}(\{1, 2, \dots, |{\bf s}|-1\}\setminus P({\bf s})).
\end{equation*}
It is easy to verify that $({\bf s}^{*})^{*}={\bf s}$.
Moreover, by \cite[Theorem 6.7]{Ho:03} we have that
\begin{equation*}
S_{p-1}({\bf s})\equiv -S_{p-1}({\bf s}^{*}) \pmod{p}
\end{equation*}
for any prime $p$.
In particular, for any positive integers $a,b$ it follows that
\begin{align*}
S_{p-1}(\{2\}^a,1,\{2\}^b) &\equiv -S_{p-1}(1,\{2\}^{a-1},3,\{2\}^{b-1},1)\pmod{p},\\
S_{p-1}(\{2\}^a,3,\{2\}^b)&\equiv -S_{p-1}(1,\{2\}^{a},1,\{2\}^{b},1) \pmod{p},\\
S_{p-1}(\{1\}^a,2,\{1\}^b)&\equiv -S_{p-1}(a+1,b+1) \pmod{p},
\end{align*}
and by Theorem \ref{T3.3.1} and Theorem \ref{T3.3.2} we are able to find $S_{p-1}(1,\{2\}^{a},1,\{2\}^{b},1)$ and
 $S_{p-1}(1,\{2\}^{a-1},3,\{2\}^{b-1},1)$  modulo a prime $p$ in terms of Bernoulli numbers.
\end{remark}

\section{Congruences for ``odd'' multiple harmonic sums}
\label{S5}

In this section we  prove some congruences involving ``odd'' multiple harmonic sums. These results will be needed in the next sections.

\begin{lemma} \label{l01}
Let ${\bf s}\in {\mathbb N}^r.$ Then for any prime $p>3,$
\begin{equation*}
\HH_{\frac{p-1}{2}}({\bf s})\equiv \frac{(-1)^{|{\bf s}|}}{2^{|{\bf s}|}} H_{\frac{p-1}{2}}(\overline{\bf s}) \pmod{p}.
\end{equation*}
\end{lemma}
\begin{proof} Changing the order of summation in the definition of $\HH_{n}({\bf s}),$ with $n=\frac{p-1}{2},$ we obtain
\begin{equation*}
\begin{split}
\HH_{n}(s_1, s_2, \ldots, s_r)&=\sum_{0\le k_1<\ldots<k_r<n}\frac{1}{(2k_1+1)^{s_1}\cdots (2k_r+1)^{s_r}}\\
&=\sum_{n>k_1>\ldots>k_r\ge 0}\frac{1}{(2(n-1-k_1)+1)^{s_1}\cdots (2(n-1-k_r)+1)^{s_r}}\\
&\equiv \sum_{0\le k_r<\ldots<k_1\le n-1}\frac{(-1)^{s_1+\cdots+s_r}}{2^{s_1+\cdots+s_r}(k_1+1)^{s_1}\cdots(k_r+1)^{s_r}}\\
&=\frac{(-1)^{|{\bf s}|}}{2^{|{\bf s}|}}H_n(s_r,\ldots,s_1) \pmod{p}.
\end{split}
\end{equation*}
\end{proof}

\begin{theorem} \label{t3.1}
For any positive integer $m$ and any prime $p>2m+1,$ we have
\begin{align}
\HH_{\frac{p-1}{2}}(2m)&\equiv \frac{m}{4^m(2m+1)}\,p\,B_{p-2m-1} \pmod{p^2}, \label{eq10}\\
\HH_{\frac{p-1}{2}}(\{2\}^m)&\equiv \frac{(-1)^{m-1}}{4^m (2m+1)}\,p\,B_{p-2m-1} \pmod{p^2} \label{eq11}.
\end{align}
\end{theorem}
\begin{proof} Let $n=(p-1)/2.$ It is easy to deduce (see \cite[Lemma 2]{HHT:11}) that
\begin{equation*}
\HH_n(2m)=H_{2n}(2m)-\frac{H_n(2m)}{2^{2m}},
\end{equation*}
which by (i) and (iv) implies the first congruence.

From \cite[Section 2.4, Lemma 2.12]{Zhao:06} it follows that the ``odd'' multiple harmonic sum $\HH_n$
as well as $H_n$ satisfies the following stuffle relation:
\begin{equation} \label{eq12}
m\HH_n(\{2\}^m)=\sum_{k=1}^m(-1)^{k-1}\HH_n(2k)\cdot \HH_n(\{2\}^{m-k}).
\end{equation}
Now the  congruence (\ref{eq11}) easily follows from the above identity and (\ref{eq10}) by induction on $m.$
Indeed, for $m=1$ it follows from (\ref{eq10}), and if $m>1,$ then by (\ref{eq12}) and induction hypothesis, we have
\begin{equation*}
m\HH_n(\{2\}^m)\equiv (-1)^{m-1}\HH_n(2m) \pmod{p^2},
\end{equation*}
and the theorem is proved.
\end{proof}

\begin{theorem} \label{t4}
Let $a, b$ be non-negative integers and a prime $p>2a+2b+3.$ Then
\begin{align*}
S_{\frac{p-1}{2}}(\{2\}^a,3,\{2\}^b)&\equiv -\frac{B_{p-2a-2b-3}}{b+1}\binom{2a+2b+2}{2a+1} \pmod{p},\\
H_{\frac{p-1}{2}}(\{2\}^a,3,\{2\}^b)&\equiv (-1)^{a+b+1}\frac{B_{p-2a-2b-3}}{a+1}\binom{2a+2b+2}{2a+1} \pmod{p}.
\end{align*}
\end{theorem}
\begin{proof}
Setting $n=\frac{p-1}{2}$ in identity (\ref{S2a32b}) and observing that
\begin{equation} \label{k1}
\frac{\binom{n}{k}}{\binom{n+k}{k}}=(-1)^k
\frac{\bigl(\frac{1}{2}-\frac{p}{2}\bigr)\bigl(\frac{3}{2}-\frac{p}{2}\bigr)\cdots\bigl(\frac{2k-1}{2}-\frac{p}{2}\bigr)}%
{\bigl(\frac{1}{2}+\frac{p}{2}\bigr)\bigl(\frac{3}{2}+\frac{p}{2}\bigr)\cdots\bigl(\frac{2k-1}{2}+\frac{p}{2}\bigr)}\equiv (-1)^k \pmod{p}
\end{equation}
we obtain
\begin{equation*}
S_{\frac{p-1}{2}}(\{2\}^a,3,\{2\}^b)\equiv -2H_{\frac{p-1}{2}}(2a+2b+3)-4H_{\frac{p-1}{2}}(2a+1,2b+2) \pmod{p}.
\end{equation*}
Now the required congruence easily follows by (iv) and (v). From Lemma \ref{l01} and Theorem \ref{t3.1} it follows that
\begin{equation*}
H_{\frac{p-1}{2}}(\{2\}^m)\equiv 2^{2m}\,\HH_{\frac{p-1}{2}}(\{2\}^m)\equiv 0 \pmod{p}.
\end{equation*}
Therefore applying identity (\ref{SH}) with $n=\frac{p-1}{2}$ and ${\bf s}=(\{2\}^a,3,\{2\}^b)$ we obtain
\begin{equation*}
(-1)^{a+b+1}S_{\frac{p-1}{2}}(\{2\}^b,3,\{2\}^a)\equiv -H_{\frac{p-1}{2}}(\{2\}^a,3,\{2\}^b) \pmod{p}
\end{equation*}
and the theorem is proved.
\end{proof}

\begin{theorem} \label{t5}
Let $a, b$ be non-negative integers, not both zero, and a prime $p>2a+2b+1.$ Then
\begin{align*}
S_{\frac{p-1}{2}}(\{2\}^a,1,\{2\}^b)&\equiv \frac{2^{1-2a-2b}-2}{2b+1}\binom{2a+2b}{2a}B_{p-2a-2b-1} \pmod{p},\\
H_{\frac{p-1}{2}}(\{2\}^a,1,\{2\}^b)&\equiv \frac{(-1)^{a+b}(2^{1-2a-2b}-2)}{2a+1}\binom{2a+2b}{2a}B_{p-2a-2b-1} \pmod{p}.
\end{align*}
\end{theorem}
\begin{proof} If $a\ge 1,$ then from identity (\ref{eq35}) with $n=(p-1)/2$ and congruence (\ref{k1}) we  find
\begin{equation*}
S_{\frac{p-1}{2}}(\{2\}^a,1,\{2\}^b)\equiv -2H_{\frac{p-1}{2}}(2a+2b+1)-4H_{\frac{p-1}{2}}(-2a,-2b-1) \pmod{p}.
\end{equation*}
Now the required result easily follows by (iv) and Lemma \ref{l02}. If $a=0,$ then from (\ref{eq34}) we have
\begin{equation} \label{eq38}
S_{\frac{p-1}{2}}(1,\{2\}^b)\equiv 2H_{\frac{p-1}{2}}(-2b-1) \pmod{p}.
\end{equation}
Moreover
\begin{equation*}
\begin{split}
H_{\frac{p-1}{2}}(-2b-1)&=\sum_{k=1}^{(p-1)/2}\frac{(-1)^k}{k^{2b+1}}=\!\!\sum_{k=(p+1)/2}^{p-1}
\frac{(-1)^{p-k}}{(p-k)^{2b+1}} \\
&\equiv \!\sum_{k=(p+1)/2}^{p-1}\frac{(-1)^k}{k^{2b+1}}=H_{p-1}(-2b-1)-H_{\frac{p-1}{2}}(-2b-1) \pmod{p}.
\end{split}
\end{equation*}
Therefore, since $b>0$, by (vi) it follows that
\begin{equation} \label{eq39}
H_{\frac{p-1}{2}}(-2b-1)\equiv \frac{1}{2}H_{p-1}(-2b-1)\equiv \frac{2^{-2b}-1}{2b+1}\,B_{p-2b-1} \pmod{p}.
\end{equation}
From (\ref{eq38}) and (\ref{eq39}) we conclude the truth of the first congruence for $a=0.$
From identity (\ref{SH}) we easily obtain
\begin{equation*}
(-1)^{a+b}S_{\frac{p-1}{2}}(\{2\}^b,1,\{2\}^a)\equiv H_{\frac{p-1}{2}}(\{2\}^a,1,\{2\}^b) \pmod{p}
\end{equation*}
and the proof is complete.
\end{proof}

\begin{lemma} \label{l3.2}
Let $r, a$ be positive integers. Then for any prime $p>r+2,$ we have
\begin{equation} \label{eq20}
\begin{split}
H_{\frac{p-1}{2}}(r)
&\equiv H_{\lfloor{\frac{p}{4}}\rfloor}(r)+
(-2)^r\sum_{k=0}^a\binom{r-1+k}{k} H_{\frac{p-1}{2}}(r+k)p^k\\
&\qquad -(-1)^r\sum_{k=0}^a\binom{r-1+k}{k}\frac{H_{\lfloor{\frac{p}{4}\rfloor}}(r+k)}{2^k}\, p^k \pmod{p^{a+1}}.
\end{split}
\end{equation}
If $r$ is a non-negative integer, then for any prime $p>2r+3,$ we have
\begin{equation*}
\HH_{\frac{p-1}{2}}(-2r-1)\equiv \frac{(-1)^{\frac{p+1}{2}}}{4^{2r+1}}\,H_{\frac{p-1}{2}}(2r+1) \pmod{p^2}.
\end{equation*}
\end{lemma}
\begin{proof} Let $n=(p-1)/2,$ $m=\lfloor{p/4\rfloor}.$ Then
the first congruence follows easily from the equality
\begin{equation*}
H_n(r)=H_m(r)+2^r\Bigl(\sum_{k=1}^n\frac{1}{(p-k)^r}-\sum_{k=1}^m\frac{1}{(p-2k)^r}\Bigr).
\end{equation*}
To prove the second one, we have
\begin{equation*}
\begin{split}
(-1)^n\sum_{k=0}^{n-1}&\frac{(-1)^k}{(2k+1)^{2r+1}}=\sum_{k=1}^n\frac{(-1)^k}{(p-2k)^{2r+1}} \\
&\qquad\quad=2\sum_{k=1}^m\frac{1}{(p-4k)^{2r+1}}-\sum_{k=1}^n\frac{1}{(p-2k)^{2r+1}} \\
&\qquad\quad=2\sum_{k=1}^m\frac{1}{(4k)^{2r+1}(p/(4k)-1)^{2r+1}}-\sum_{k=1}^n
\frac{1}{(2k)^{2r+1}(p/(2k)-1)^{2r+1}} \\
&\qquad\quad\equiv -\frac{1}{2^{4r+1}}\Bigl(H_m(2r+1)+\frac{(2r+1)p}{4}\,H_m(2r+2)\Bigr) \\
&\qquad\quad \, +
\frac{1}{2^{2r+1}}\,H_n(2r+1) \pmod{p^2}.
\end{split}
\end{equation*}
From (\ref{eq20}) with $a=1$ and $r$ replaced by $2r+1,$ we get
\begin{equation*}
2H_m(2r+1)+\frac{(2r+1)p}{2}H_m(2r+2)\equiv (1+2^{2r+1})H_n(2r+1) \pmod{p^2},
\end{equation*}
which after substitution into the right-hand side of the above congruence implies the result.
\end{proof}

\section{Results on $p$-analogues of Leshchiner's series}

In this section, we consider finite $p$-analogues of zeta and beta values arising from the truncation
of Leshchiner's series (\ref{eq02})--(\ref{eq04}) and  defined by the finite sums
\begin{equation*}
\zeta_p(2m+2):=\frac{3}{2}\sum_{k=1}^{p-1}\frac{(-1)^mH_{k-1}(\{2\}^m)}{k^2\binom{2k}{k}}
+2\sum_{j=1}^m\sum_{k=1}^{p-1}\frac{(-1)^{m-j}H_{k-1}(\{2\}^{m-j})}{k^{2j+2}\binom{2k}{k}},
\end{equation*}
\begin{equation*}
\begin{split}
\zeta_p(2m+3)&:=\frac{5}{2}\sum_{k=1}^{p-1}\frac{(-1)^{k+m-1}H_{k-1}(\{2\}^m)}{k^3\binom{2k}{k}}\\
& +2\sum_{j=1}^m\sum_{k=1}^{p-1}\frac{(-1)^{k+m-j-1}H_{k-1}(\{2\}^{m-j})}{k^{2j+3}\binom{2k}{k}},
\end{split}
\end{equation*}
\begin{equation*}
\begin{split}
\overline{\zeta}_p(2m+2)&:=\frac{5}{4}\sum_{k=0}^{(p-3)/2}\frac{(-1)^{k+m}\binom{2k}{k}\HH_k(\{2\}^m)}{16^k(2k+1)^2}\\
& +\sum_{j=1}^m\sum_{k=0}^{(p-3)/2}\frac{(-1)^{k+m-j}\binom{2k}{k}\HH_k(\{2\}^{m-j})}{16^k(2k+1)^{2j+2}},
\end{split}
\end{equation*}
\begin{equation*}
\begin{split}
\beta_p(2m+1)&:=\frac{3}{4}\sum_{k=0}^{(p-3)/2}\frac{(-1)^m\binom{2k}{k}\HH_k(\{2\}^m)}{16^k(2k+1)}\\
&+\sum_{j=1}^m\sum_{k=0}^{(p-3)/2}\frac{(-1)^{m-j}\binom{2k}{k}\HH_k(\{2\}^{m-j})}{16^k(2k+1)^{2j+1}}.
\end{split}
\end{equation*}
Applications of our results on congruences obtained in Sections \ref{S4} and \ref{S5} allowed us to establish the following
congruences for the finite $p$-analogues of Leshchiner's series.

\begin{theorem} \label{t2.1}
Let $m$ be a non-negative integer and $p$ be a prime such that $p>2m+3.$ Then
\begin{align*}
2p\,\zeta_p(2m+2)&\equiv -\frac{4m(1-2^{-2m})}{2m+1}\, B_{p-2m-1} \pmod{p}, \\
\zeta_p(2m+3)&\equiv -\frac{(m+1)(2m+1)}{2m+3}\, B_{p-2m-3} \pmod{p}.
\end{align*}
\end{theorem}
Note that for $\zeta_p(2),$ $\zeta_p(3),$ $\zeta_p(4),$ stronger congruences were proved by the authors in \cite{Ta:10, HH:11}.
For $\zeta_p(5),$ we recover the congruence $\zeta_p(5)\equiv -6/5 B_{p-5} \pmod{p}$ proved in \cite{HH:11}.

\begin{theorem} \label{t2.2}
Let $m$ be a non-negative integer and $p$ be a prime such that $p>2m+3.$ Then
\begin{align*}
\beta_p(2m+1)&\equiv \frac{(-1)^{(p+1)/2}m(4^m-1)}{16^{m}(2m+1)}\,B_{p-2m-1}
 \pmod{p}, \\
\overline{\zeta}_p(2m+2)&\equiv -\frac{2m^2+3m+2}{2^{2m+3}(2m+3)}\,p\,B_{p-2m-3}
 \pmod{p^2}.
\end{align*}
\end{theorem}
Note that for $\beta_p(1)$ and $\overline{\zeta}_p(2),$ sharper congruences were proved in \cite{HHT:11}.

\section{Finite identities and their formal proofs}

By revisiting the combinatorial proof (due to D.~Zagier) presented in Section 5 in \cite{Le:81}, it is possible to obtain
 finite versions of  identities (\ref{eq02})--(\ref{eq04}). Here we derive these identities  in a formal unified way
based on applications of appropriate WZ pairs.
\begin{lemma} \label{l4.1}
Let $m$ be a non-negative integer. Then for every positive integer $n,$ we have
\begin{equation} \label{eq13}
\begin{split}
&\qquad\sum_{k=1}^n\frac{(-1)^{k-1}}{k^{2m+2}}
-\frac{1}{2}\sum_{k=1}^n\frac{(-1)^{n+k+m}H_{k-1}(\{2\}^m)}{k^2\binom{n}{k}\binom{n+k}{k}} \\[2pt]
&\qquad\quad=\frac{3}{2}\sum_{k=1}^n\frac{(-1)^mH_{k-1}(\{2\}^m)}{k^2\binom{2k}{k}}
+2\sum_{j=1}^m\sum_{k=1}^n\frac{(-1)^{m-j}H_{k-1}(\{2\}^{m-j})}{k^{2j+2}\binom{2k}{k}},
\end{split}
\end{equation}
\begin{equation} \label{eq14}
\begin{split}
&\qquad\,\,\,\,\sum_{k=1}^n\frac{1}{k^{2m+3}}-\frac{1}{2}\sum_{k=1}^n\frac{(-1)^{k+m}H_{k-1}(\{2\}^m)}{k^3\binom{n}{k}\binom{n+k}{k}} \\[2pt]
&\qquad\quad\,\,\,\,=\frac{5}{2}\sum_{k=1}^n\frac{(-1)^{k+m-1}H_{k-1}(\{2\}^m)}{k^3\binom{2k}{k}}
+2\sum_{j=1}^m\sum_{k=1}^n\frac{(-1)^{k+m-j-1}H_{k-1}(\{2\}^{m-j})}{k^{2j+3}\binom{2k}{k}}.
\end{split}
\end{equation}
\end{lemma}
\begin{proof} For non-negative integers $n, k,$ we consider a pair of functions
\begin{align*}
F_1(n,k)&=\frac{(-1)^{n+k}(n-k-1)!(1+a)_k(1-a)_k}{2(n+k+1)!}, \qquad n\ge k+1, \\
G_1(n,k)&=\frac{(-1)^{n+k}(n-k)!(n+1)(1+a)_k(1-a)_k}{((n+1)^2-a^2)(n+k+1)!}, \qquad n\ge k,
\end{align*}
where $(c)_k$ is the shifted factorial defined by $(c)_k=c(c+1)\cdots (c+k-1),$ $k\ge
1,$ and $(c)_0=1.$
It is easy to check that $(F_1, G_1)$ is a WZ pair, i.e., it  satisfies the relation
\begin{equation*}
F_1(n+1,k)-F_1(n,k)=G_1(n,k+1)-G_1(n,k).
\end{equation*}
Summing the above equality first over $k$ and then over $n,$ we get the following summation formula (see \cite[Section 4]{HH:11}):
\begin{equation} \label{eq17}
\sum_{n=1}^NG_1(n-1,0)=\sum_{n=1}^N(G_1(n-1,n-1)+F_1(n,n-1))-\sum_{k=1}^NF_1(N,k-1), \quad N\in {\mathbb N}.
\end{equation}
Substituting $(F_1, G_1)$ in (\ref{eq17}) we obtain
\begin{equation} \label{eq18}
\begin{split}
\sum_{n=1}^N\frac{(-1)^{n-1}}{n^2-a^2}&=\frac{1}{2}\sum_{n=1}^N\frac{1}{n^2\binom{2n}{n}}\frac{3n^2+a^2}{n^2-a^2}
\prod_{m=1}^{n-1}\Bigl(1-\frac{a^2}{m^2}\Bigr) \\[2pt]
&+
\frac{1}{2}\sum_{k=1}^N\frac{(-1)^{N+k}}{k^2\binom{N}{k}\binom{N+k}{k}}\prod_{m=1}^{k-1}
\Bigl(1-\frac{a^2}{m^2}\Bigr).
\end{split}
\end{equation}
Expanding both sides of (\ref{eq18}) in powers
of $a^2$ and comparing coefficients of $a^{2m}$ we derive the finite identity (\ref{eq13}). Similarly, considering the WZ pair
\begin{equation*}
F_2(n,k)=\frac{(-1)^n}{k+1}\,F_1(n,k), \qquad G_2(n,k)=\frac{(-1)^n}{n+1}\, G_1(n,k)
\end{equation*}
and applying the same argument as above we obtain the second identity.
\end{proof}
\begin{lemma} \label{l4.2}
For any non-negative integers $m, n$ we have
\begin{equation} \label{eq15}
\begin{split}
&\qquad\sum_{k=0}^n\frac{(-1)^{k}}{(2k+1)^{2m+1}}-\frac{1}{4}\sum_{k=0}^n\frac{(-1)^{n+k+m}\binom{2k}{k}\HH_{k}(\{2\}^m)}{16^k(2k+1)\binom{n+k+1}{2k+1}}
\\[5pt]
&\qquad\quad=\frac{3}{4}\sum_{k=0}^n\frac{(-1)^m\binom{2k}{k}
\HH_{k}(\{2\}^m)}{16^k(2k+1)}
+\sum_{j=1}^m\sum_{k=0}^n\frac{(-1)^{m-j}\binom{2k}{k}\HH_{k}(\{2\}^{m-j})}{16^k(2k+1)^{2j+1}},
\end{split}
\end{equation}
\begin{equation} \label{eq16}
\begin{split}
&\qquad\qquad\sum_{k=0}^n\frac{1}{(2k+1)^{2m+2}}
+\frac{1}{4}\sum_{k=0}^n\frac{(-1)^{k+m}\binom{2k}{k}\HH_{k}(\{2\}^m)}{16^k(2k+1)^2\binom{n+k+1}{2k+1}} \\[5pt]
&\qquad\qquad\,\,\,\,=\frac{5}{4}\sum_{k=0}^n\frac{(-1)^{k+m}\binom{2k}{k}\HH_{k}(\{2\}^m)}{16^k(2k+1)^2}
+\sum_{j=1}^m\sum_{k=0}^n\frac{(-1)^{k+m-j}\binom{2k}{k}\HH_{k}(\{2\}^{m-j})}{16^k(2k+1)^{2j+2}}.
\end{split}
\end{equation}
\end{lemma}
\begin{proof}
As in the proof of Lemma \ref{l4.1}, applying summation formula (\ref{eq17}) to the WZ pairs
\begin{equation*}
F_1(n,k)=\frac{(-1)^{n+k}(n-k-1)!}{4(n+k)!}\,\Bigl(\frac{1+a}{2}\Bigr)_k\Bigl(\frac{1-a}{2}\Bigr)_k,
\quad n\ge k+1,
\end{equation*}
\begin{equation*}
G_1(n,k)=\frac{(-1)^{n+k}(n-k)! (2n+1)}{(n+k)!((2n+1)^2-a^2)}\,\Bigl(\frac{1+a}{2}\Bigr)_k\Bigl(\frac{1-a}{2}\Bigr)_k,
\quad n\ge k,
\end{equation*}
and
\begin{equation*}
F_2(n,k)=\frac{(-1)^n}{2k+1}\,F_1(n,k), \qquad G_2(n,k)=\frac{(-1)^n}{2n+1}\,G_1(n,k),
\end{equation*}
we get the required identities.
\end{proof}

\section{Proofs of Theorems \ref{t2.1} and \ref{t2.2}}

{\bf Proof of Theorem \ref{t2.1}}
Setting $n=p-1$ in (\ref{eq13}), we get
\begin{equation*}
\zeta_p(2m+2)=H_{p-1}(-2m-2)-\frac{1}{2}\sum_{k=1}^{p-1}\frac{(-1)^{k+m}H_{k-1}(\{2\}^m)}{k^2\binom{p-1}{k}\binom{p-1+k}{k}}.
\end{equation*}
Since
\begin{equation} \label{eq19}
\begin{split}
\frac{(-1)^k}{k^2\binom{p-1}{k}\binom{p-1+k}{k}}&=\frac{1}{pk}\prod_{m=1}^k\Bigl(1-\frac{p}{m}\Bigr)^{-1}
\prod_{m=1}^{k-1}\Bigl(1+\frac{p}{m}\Bigr)^{-1} \\[2pt]
&\equiv \frac{1}{pk}+\frac{1}{k^2}+\frac{pH_k(2)}{k} \pmod{p^2},
\end{split}
\end{equation}
we obtain
\begin{equation*}
\begin{split}
\zeta_p(2m+2)&\equiv -H_{p-1}(-2m-2)-\frac{(-1)^m}{2p}\,H_{p-1}(\{2\}^m,1) \\
&\quad-\frac{(-1)^m}{2}H_{p-1}(\{2\}^{m+1}) \pmod{p},
\end{split}
\end{equation*}
which by (i), (vi), and Theorem \ref{T3.3.2} implies the required congruence.

\noindent Similarly, setting $n=p-1$ in (\ref{eq14}) and applying (\ref{eq19}), (i) and Theorem \ref{T3.3.1} we get
\begin{equation*}
\begin{split}
\zeta_p(2m+3)&=H_{p-1}(2m+3)-\frac{1}{2}\sum_{k=1}^{p-1}\frac{(-1)^{k+m}H_{k-1}(\{2\}^m)}{k^3\binom{p-1}{k}
\binom{p-1+k}{k}} \\
&\equiv H_{p-1}(2m+3)+\frac{(-1)^{m-1}}{2}\Bigl(\frac{1}{p}H_{p-1}(\{2\}^{m+1})+
H_{p-1}(\{2\}^m,3)\Bigr) \\
&\equiv -\frac{B_{p-2m-3}}{2m+3}+\frac{(-1)^{m-1}}{2}\,H_{p-1}(\{2\}^m,3)\\[3pt]
&\equiv -\frac{(m+1)(2m+1)}{2m+3}B_{p-2m-3} \pmod{p}.
\end{split}
\end{equation*}

{\bf Proof of Theorem \ref{t2.2}}

Setting $n=(p-3)/2$ in (\ref{eq15}), we get
\begin{equation*}
\beta_p(2m+1)=\HH_{\frac{p-1}{2}}(-2m-1)+\frac{(-1)^{(p-1)/2+m}}{4}
\sum_{k=0}^{(p-3)/2}\frac{(-1)^k\binom{2k}{k}\HH_k(\{2\}^m)}{16^k(2k+1)\binom{(p-1)/2+k}{2k+1}}.
\end{equation*}
Taking into account that
\begin{equation} \label{eq21}
\frac{\binom{2k}{k}}{(-16)^k\binom{(p-1)/2+k}{2k+1}}=
\frac{2(2k+1)}{p-1-2k}\,\prod_{j=0}^{k-1}\Bigl(1-\frac{p^2}{(2j+1)^2}\Bigr)^{-1}
\equiv -2-\frac{2p}{2k+1} \pmod{p^2},
\end{equation}
we obtain by Theorem \ref{t3.1},
\begin{equation*}
\begin{split}
\beta_p(2m+1)&\equiv \HH_{\frac{p-1}{2}}(-2m-1)+\frac{(-1)^{(p+1)/2+m}}{2}
\sum_{k=0}^{(p-3)/2}\frac{\HH_k(\{2\}^m)}{2k+1}\Bigl(1+\frac{p}{2k+1}\Bigr)\\
&\equiv \HH_{\frac{p-1}{2}}(-2m-1)+\frac{(-1)^{(p+1)/2+m}}{2}\,\HH_{\frac{p-1}{2}}(\{2\}^m,1) \pmod{p^2},
\end{split}
\end{equation*}
which implies the first congruence of the theorem by Lemmas \ref{l3.2} and \ref{l01}, and Theorem \ref{t5}.

Similarly, from (\ref{eq16}) by (\ref{eq21}), Theorem \ref{t3.1}, Lemma \ref{l01} and Theorem \ref{t4},  we have
\begin{equation*}
\begin{split}
\overline{\zeta}_p(2m+2)&\equiv \HH_{\frac{p-1}{2}}(2m+2)+\frac{(-1)^{m-1}}{2}
\Bigl(\HH_{\frac{p-1}{2}}(\{2\}^{m+1})+p\HH_{\frac{p-1}{2}}(\{2\}^m,3)\Bigr) \\
&\equiv \frac{2m+1}{2^{2m+3}(2m+3)}\,p\,B_{p-2m-3}+\frac{(-1)^{m-1}}{2}\,p\,\HH_{\frac{p-1}{2}}(\{2\}^m,3) \\
&\equiv -\frac{2m^2+3m+2}{2^{2m+3}(2m+3)}\,pB_{p-2m-3}
\pmod{p^2}.
\end{split}
\end{equation*}

\vspace{0.2cm}

{\bf\small Acknowledgements.} The authors wish to thank the referees of the journal whose valuable remarks helped to improve the presentation of the paper.
This work was done while the first and second authors were on sabbatical in the Department of Mathematics and Statistics at Dalhousie University, Halifax, Canada.
The authors wish to thank the Chair of the Department, Professor~Karl Dilcher, for the hospitality, excellent working conditions, and useful discussions.
The third author would like to thank Sandro Mattarei for the valuable help in
improving the presentation of this paper.

\bibliographystyle{amsplain}

\end{document}